\documentclass[12pt,3p]{elsarticle}
\usepackage{amsmath, amsfonts}
\usepackage{amsthm}
\usepackage[toc,page]{appendix}
\usepackage{enumitem}
\usepackage{mathtools}
\usepackage{amssymb}
\usepackage{hyperref}
\usepackage{natbib}
\usepackage{etoolbox}
\patchcmd{\thmhead}{(#3)}{#3}{}{}
\usepackage{graphicx}
\numberwithin{equation}{section}
\usepackage{graphicx}
\newtheorem{thm}{Theorem}[section]
\newtheorem*{theorem*}{Internal Edge Homogenization Result}

\newtheorem{lem}[thm]{Lemma}
\newtheorem{prop}[thm]{Proposition}
\theoremstyle{definition}

\theoremstyle{remark}
\newtheorem{rem}[thm]{Remark}
\usepackage{fancyhdr}
\makeatletter
\def\ps@pprintTitle{%
	\let\@oddhead\@empty
	\let\@evenhead\@empty
	\def\@oddfoot{}%
	\let\@evenfoot\@oddfoot}
\makeatother
\begin{document}
	\begin{frontmatter}
		\title{Generic Simplicity of Spectral Edges and Applications to Homogenization}		
		\author{Sivaji Ganesh Sista}
		\ead{sivaji.ganesh@iitb.ac.in}
		\cortext[cor1]{Corresponding author}
		\author{Vivek Tewary}
		\ead{vivekt@iitb.ac.in}
		\address{Department of Mathematics, Indian Institute of Technology Bombay, Powai, Mumbai, 400076, India.}
		\begin{abstract}
			We consider the spectrum of a second-order elliptic operator in divergence form with periodic coefficients, which is known to be completely described by Bloch eigenvalues. We show that under small perturbations of the coefficients, a multiple Bloch eigenvalue can be made simple. The Bloch wave method of homogenization relies on the regularity of spectral edge. The spectral tools that we develop, allow us to obtain simplicity of an internal spectral edge through perturbation of the coefficients. As a consequence, we are able to establish Bloch wave homogenization at an internal edge in the presence of multiplicity by employing the perturbed Bloch eigenvalues. We show that all the crossing Bloch modes contribute to the homogenization at the internal edge and that higher and lower modes do not contribute to the homogenization process.
		\end{abstract}
		
		\begin{keyword}
			Bloch eigenvalues \sep Genericity \sep Periodic Operators \sep Homogenization
			\MSC[2010] 47A55 \sep 35J15 \sep 35B27
		\end{keyword}
		
	\end{frontmatter}
\section{Introduction}

The goal of the paper is to study regularity properties of spectral edges of a periodic second-order elliptic operator in divergence form, given by
\begin{equation}
\mathcal{A}u:=-\frac{\partial}{\partial y_k}\left(a_{kl}(y)\frac{\partial u}{\partial y_l}\right),\label{eq1:operator}
\end{equation} where summation over repeated indices is assumed.
We make the following assumptions on the coefficients of the operator~\eqref{eq1:operator}: The coefficients $a_{kl}(y)$ are measurable bounded real-valued periodic functions defined on $\mathbb{R}^d$. Let $Y=[0,2\pi)^d$ be a basic cell for its lattice of periods in the $d$-dimensional euclidean space $\mathbb{R}^d$. The space of measurable bounded periodic real-valued functions in $Y$ is denoted by $L^\infty_\sharp(Y,\mathbb{R})$. Hence, $a_{kl}\in L^\infty_\sharp(Y,\mathbb{R})$. In many instances, we will identify $Y$ with a torus $\mathbb{T}^d$ and the space $L^\infty_\sharp(Y,\mathbb{R})$ with $L^\infty(\mathbb{T}^d,\mathbb{R})$, in the standard way. The matrix $A=(a_{kl})$ is symmetric, i.e., $a_{kl}(y)=a_{lk}(y)$. Further, the matrix $A$ is {\it coercive}, i.e., there exists an $\alpha>0$ such that  
\begin{equation}\label{coercivity}
\forall\, v\in\mathbb{R}^d\mbox{ and } a.e.\, y\in\mathbb{R}^d,\langle A(y)v, v\rangle\geq \alpha||v||^2.
\end{equation}

Let $Y^{'}=\left[-\dfrac{1}{2},\dfrac{1}{2}\right)^d$ be a basic cell for the dual lattice in $\mathbb{R}^d$. Then, the spectrum of $\mathcal{A}$ can be studied by evaluating, for $\eta\in Y'$, the spectrum of the shifted operator \begin{align}\mathcal{A}(\eta)=e^{-i\eta\cdot y}\mathcal{A}e^{i\eta\cdot y}=-\left(\frac{\partial}{\partial y_k}+i\eta_k\right)a_{kl}(y)\left(\frac{\partial}{\partial y_l}+i\eta_l\right).\label{shiftedoperator}\end{align}

This is an unbounded operator in $L^2_\sharp(Y)$, the space of all $L^2_{loc}(\mathbb{R}^d)$ functions that are $Y$-periodic. The operator $\mathcal{A}$ in $L^2(\mathbb{R}^d)$ is unitarily equivalent to the fibered operator 
\begin{equation*}
\int^{\bigoplus}_{Y^{'}}\mathcal{A}(\eta)d\eta
\end{equation*}
in the Bochner space $L^2(Y^{'},L^2_{\sharp}(Y))$. As a consequence of this fact, the spectrum of $\mathcal{A}$ is the union of the spectra of $\mathcal{A}(\eta)$ in $L^2_{\sharp}(Y)$ as $\eta$ varies in $Y^{'}$. For a proof, see~\cite[p.~284]{Reed1978}. Let $(\lambda_n(\eta))_{n=1}^\infty$ denote the sequence of increasing eigenvalues for $\mathcal{A}(\eta)$, counting multiplicity. The functions $\eta\mapsto\lambda_n(\eta)$ are known as the Bloch eigenvalues of the operator $\mathcal{A}$. Let $\sigma_n^-=\displaystyle\min_{\eta\in Y^{'}}\lambda_n(\eta)$ and $\sigma_n^+=\displaystyle\max_{\eta\in Y^{'}}\lambda_n(\eta)$, then, the spectrum of the operator $\mathcal{A}$ is given by $\bigcup_{n\in\mathbb{N}}[\sigma_n^-,\sigma_n^+]$. Therefore, it is a union of closed intervals, which may overlap. However, it may also be written as $[0,\infty)\setminus\sqcup_{j=1}^N (\mu_j^-,\mu_j^+)$, where $N$ takes values in $\mathbb{N}\cup\{\infty\}$. The pairwise disjoint intervals $(\mu_j^-,\mu_j^+)$ are known as spectral gaps and $(\mu_j^{\pm})_{j=1}^N$ are known as spectral edges. As depicted in Fig.~\ref{figure1}, $\sigma_n^\pm$ may not be spectral edges, even though the corresponding Bloch eigenvalue is simple.

In the first part of the paper, we will study regularity of Bloch eigenvalues near the points where the spectral edge is attained in the dual parameter space. Regularity properties of the Bloch eigenvalues in the parameter are important in applications in the theory of effective mass~\cite{AllairePiatnitski2005} and Bloch wave method in homogenization~\cite{Conca1997}. For periodic Schr\"odinger operators, Wilcox~\cite{Wilcox78} proved that outside a set of measure zero in the dual parameter space, the Bloch eigenvalues are analytic and the Bloch eigenfunctions may be chosen to be analytic. However, this measure zero set might intersect the spectral edge, which would limit applicability of such a result. For the linear elasticity operator, the Bloch eigenvalues are not analytic near the bottom spectral edge, which poses a major difficulty in the passage to limit in the Bloch wave method of homogenization~\cite{SivajiGanesh2005}.

Study of parametrized eigenvalue problems is an active area of research, even in finite dimensions~\cite{michor98},~\cite{Rainer2014}. Broadly speaking, regularity results for parametrized eigenvalues of selfadjoint operators are available in two cases: (i) for one parameter eigenvalue problems~\cite{Rellich1969},~\cite{Kato1995}.
(ii) for simple eigenvalues, regardless of the number of parameters. Multiple parameters are unavoidable in most applications of interest. Examples include propagation of singularities for hyperbolic systems of equations with multiple characteristics leading to novel phenomena such as conical refraction~\cite{lax1982},~\cite{dencker1988},  stability of hyperbolic initial-boundary-value problems~\cite{Zumbrun2005} and Bloch waves for elasticity system~\cite{SivajiGanesh2005}. Hence, an assumption of simplicity is useful in applications~\cite{AllaireVanni2004},~\cite{AllaireRauch2011},~\cite{AllaireRauch2013}. 

In the literature, it has been shown that under perturbations of some relevant parameters like domain shape, coefficients, potentials etc, a multiple eigenvalue can be made simple. In a well-known paper~\cite{Albert1975}, Albert proves that, for a compact manifold $M$, the set of all smooth potentials $V\in C^{\infty}(M)$ for which the operator $-\Delta+V$ has only simple eigenvalues is a residual set in the space of all smooth admissible potentials. Similar results were proved by Uhlenbeck~\cite{Uhlenbeck1976} using topological methods. Generic simplicity of the spectrum with respect to domain has been established and applied in proving stabilizability and controllability results for the plate equation~\cite{Zuazua2000} and the Stokes system in two dimensions~\cite{Zuazua2001} by Ortega and Zuazua.

We intend to generalize Albert's method to the spectrum of periodic operators. Albert's result is applicable to operators with discrete spectrum, whereas a periodic operator typically has no eigenvalues. The symmetries of the periodic operator allow us to write it as a direct integral of operators with compact resolvent. Hence, the method of Albert may be applied in a fiberwise manner. However, the fiber~\eqref{shiftedoperator} is an operator with complex-valued coefficients. Further, the perturbation is sought in the second-order term as opposed to the zeroth-order term in~\cite{Albert1975}. In this paper, we overcome these difficulties and prove that the Bloch eigenvalues can be made simple locally in the parameter through a perturbation in the coefficients of the operator $\mathcal{A}$. Further, by applying fiberwise perturbation on the fibered operator $\int^{\bigoplus}_{Y^{'}}\mathcal{A}(\eta)d\eta$, we can make sure that the corresponding eigenvalue of interest is simple for all parameter values.
\begin{figure}
	\centering
	\includegraphics[scale=0.589]{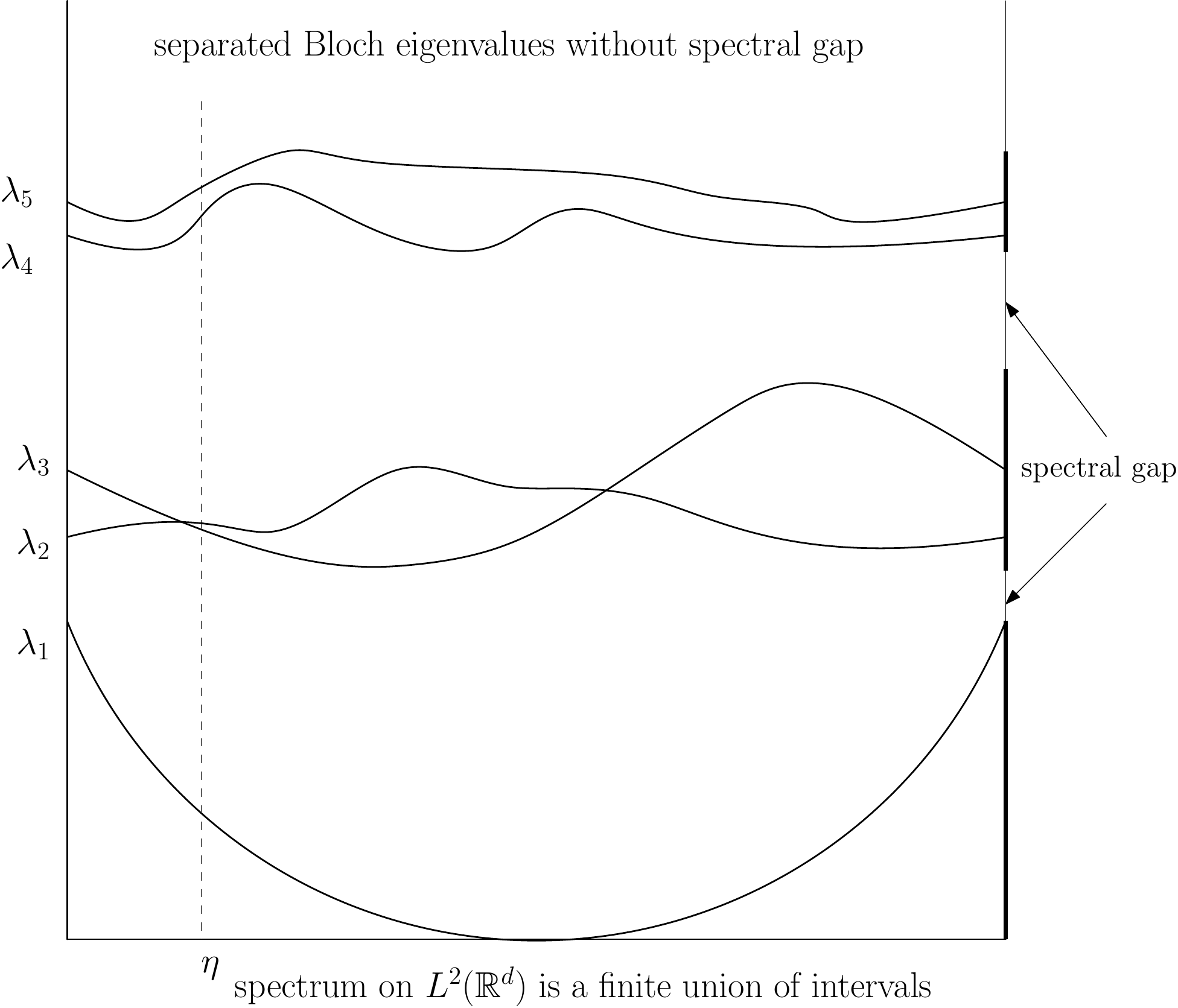}	
	\caption{Bloch eigenvalues $\lambda_4$ and $\lambda_5$ are simple, but have no spectral gap between them.}\label{figure1}
\end{figure}

The latter part of the paper is concerned with the theory of Bloch wave homogenization. In homogenization, one studies the limits of solutions to equations with highly oscillatory coefficients, such as
\begin{align}\label{anotherequation}
-\nabla\cdot\left(A\left(\frac{x}{\epsilon}\right)\nabla u^{\epsilon}\right)+\varkappa^2u^{\epsilon}=f~~\mbox{ in }\mathbb{R}^d,
\end{align}

for $f\in L^2(\mathbb{R}^d)$ and $\varkappa>0$. 

Suppose that $u^\epsilon$ converges weakly in $H^1(\mathbb{R}^d)$ to $u^{*}$. Then, the theory of homogenization~\cite{Tartar2009},~\cite{Bensoussan2011} shows that the limit $u^{*}$ solves an equation of the same type and identifies the matrix $A^*$:
\begin{align*}
-\nabla\cdot\left(A^*\nabla u^{*}\right)+\varkappa^2u^{*}=f~~\mbox{ in }\mathbb{R}^d.
\end{align*}
Bloch wave method of homogenization achieves this characterization through regularity properties of Bloch eigenvalues at the bottom of the spectrum. In particular, the homogenized matrix $A^{*}$ is characterized by the Hessian of the lowest Bloch eigenvalue at $0\in Y^{'}$~\cite{Conca1997}. Similarly, in the theory of internal edge homogenization~\cite{Birman2006}, the following regularity properties of Bloch eigenvalues near the spectral edge play an important role in obtaining operator error estimates:
\begin{enumerate}
	\item[(A)] The spectral edge must be simple, i.e., it is attained by a single Bloch eigenvalue.
	\item[(B)] The spectral edge must be attained at finitely many points by a Bloch eigenvalue.
	\item[(C)] The spectral edge must be non-degenerate, i.e., for some $m,r\in\mathbb{N}$, if the Bloch eigenvalue $\lambda_m(\eta)$ attains the spectral edge $\lambda_0$ at the points $\{\eta_j\}_{j=1}^r$, then the Bloch eigenvalue must satisfy, for $j=1,2,\ldots,r$,
	\begin{align*}
	\lambda_{m}(\eta)-\lambda_0=(\eta-\eta_j)^T B_j (\eta-\eta_j)+{O}(|\eta-\eta_j|^3), \mbox{ for }  \eta \mbox{ near } \eta_j,
	\end{align*}
    where $B_j$ are positive definite matrices.
\end{enumerate}

While these features are readily available for the lowest Bloch eigenvalue corresponding to the divergence-type scalar elliptic operator, these properties may not be available for other spectral gaps of the same operator~\cite{Kuchment}. However, the following results are available regarding these properties: Klopp and Ralston~\cite{KloppRalston00} proved the simplicity of a spectral edge of Schr\"odinger operator $-\Delta+V$ under perturbation of the potential term. In two dimensions, spectral edges are known to be isolated~\cite{Filonov15}. Also, in two dimensions, a degenerate spectral edge can be made non-degenerate through a perturbation with a potential having a larger period~\cite{ParShteren17}. 

The validity of hypotheses (A), (B), (C) is usually assumed in the literature~\cite{Kuchment}; for example, in establishing Green's function asymptotics~\cite{KuchmentRaich2012},~\cite{KhaKuchmentRaich2017}, for internal edge homogenization~\cite{Birman2006} and to establish localization for random Schr\"odinger operators~\cite{Veselic2002}. Local simplicity of Bloch eigenvalues is assumed in the study of diffractive geometric optics~\cite{AllaireRauch2011},~\cite{AllaireRauch2013} and homogenization of periodic systems~\cite{AllaireVanni2004}. Following Klopp and Ralston~\cite{KloppRalston00}, we apply a perturbation to the coefficients of the operator $\mathcal{A}$ so that a multiple spectral edge becomes simple, under the condition that the coefficients are in $W^{1,\infty}$. However, if the coefficients of $\mathcal{A}$ are in $L^\infty$, a multiple spectral edge can be made simple through a small perturbation of the coefficients with the added assumption that the spectral edge is attained at only finitely many points. Thus, our results suggest a possible interplay between the validity of these assumptions and the regularity of the coefficients. Further, these spectral tools also allow us to achieve homogenization at an internal edge in the presence of multiplicty. 

More details on the spectrum of elliptic periodic operators may be found in Reed and Simon~\cite{Reed1978} and for state of the art on  periodic differential operators, see the review by Kuchment~\cite{Kuchment}. 

\subsection{Main Results}
Let $S\!ym(d)$ denote the space of all real symmetric matrices, i.e., if \(A=(a_{kl})\in S\!ym(d)\), then $a_{kl}=a_{lk}$. Let \begin{equation*}M_B^{>}=\{A:\mathbb{R}^d\to{S\!ym}(d):a_{kl}\in L^{\infty}_\sharp(Y,\mathbb{R}) \mbox{ and $A$ is coercive }\}.\end{equation*}

$M_B^{>}$ may be identified as a subset of the space of $d(d+1)/2$-tuples of $L^{\infty}_\sharp$ functions and we shall use the norm-topology on this space in our further discussion. A Baire space is a topological space in which the countable intersection of dense open sets is dense. Note that $M_B^{>}$ is an open subset of the space of all symmetric matrices with $L^\infty_\sharp(Y,\mathbb{R})$ entries, which forms a complete metric space, and hence $M_B^{>}$ is a Baire Space. We shall call a property {\it generic} in a topological space $X$, if it holds on a set whose complement is of first category in $X$. In particular, a property that is generic on a Baire space holds on a dense set.

The rest of the subsection will be devoted to the statements of the main results.

\begin{thm}
	\label{theorem:1}
	Let $\eta_0\in Y^{'}$. The eigenvalues of the shifted operator $\mathcal{A}(\eta_0)$ are generically simple with respect to the coefficients $A=\left(a_{kl}\right)_{k,l=1}^d$ in $M_B^{>}$.
\end{thm}

\begin{rem}
		\item Theorem~\ref{theorem:1} is an extension of the theorem of Albert~\cite{Albert1975} which proves that the eigenvalues of $-\Delta+V$ are generically simple with respect to $V\in C^{\infty}(M)$ for a compact manifold $M$. The potential $V$ is the quantity of interest for Schr\"odinger operator, $-\Delta+V$. For the applications that we have in mind, for example, the theory of homogenization, the periodic matrix $A$ in the divergence type elliptic operator $-\nabla\cdot(A\nabla)$ is of physical importance. The spectrum of such operators is not discrete, and is analyzed through Bloch eigenvalues, which introduces an extra parameter $\eta\in Y^{'}$ to the problem. The determination of real-valued  perturbation for the shifted operator $\mathcal{A}(\eta)$, which has complex-valued coefficients, poses additional difficulties, when coupled with the lack of regularity of the coefficients which the applications demand.
\end{rem}

\begin{thm}\label{theorem:2}
	Let $m\in\mathbb{N}$, then for the Bloch eigenvalue $\lambda_m(\eta)$ of the periodic operator $\mathcal{A}=-\nabla\cdot(A\nabla)$, where $A\in M_B^>$, there exists a perturbation of $\mathcal{A}$ such that the perturbed eigenvalue $\tilde{\lambda}_m(\eta)$ is simple for all $\eta\in Y^{'}$.
\end{thm}

A spectral edge $\lambda_0$ is said to be {\it simple} if the set $\displaystyle\{m\in\mathbb{N}:\exists\,\eta\in Y^{'}\mbox{ such that }\lambda_m(\eta)=\lambda_0\}$ is a singleton. A spectral edge is said to be multiple if it is not simple.

\begin{thm}\label{theorem:3}
	Let $A\in M_B^>$. Further, suppose that its entries $A=\left(a_{kl}\right)_{k,l=1}^d$ belong to the class $W^{1,\infty}_\sharp(Y,\mathbb{R})$. Then, a multiple spectral edge of the operator $\mathcal{A}=\displaystyle -\nabla\cdot (A\nabla)$ can be made simple by a small perturbation in the coefficients.
\end{thm}

\begin{thm}\label{theorem:4}
	Let $A\in M_B^>$. Further, suppose that its entries $A=\left(a_{kl}\right)_{k,l=1}^d$ belong to the class $L^{\infty}_\sharp(Y,\mathbb{R})$. Let $\lambda_0$ correspond to the upper edge of a spectral gap of $\mathcal{A}$ and let $m$ be the smallest index such that the Bloch eigenvalue $\lambda_{m}$ attains $\lambda_0$. Assume that the spectral edge is attained by $\lambda_m(\eta)$ at finitely many points. Then, there exists a matrix $B=(b_{kl})_{k,l=1}^d$ with $L^\infty_\sharp(Y,\mathbb{R})$-entries and $t_0>0$ such that for every $t\in(0,t_0]$, a spectral edge is achieved by the Bloch eigenvalue $\lambda_m(\eta;A+tB)$ of the operator $\mathcal{A}=\displaystyle -\nabla\cdot (A+tB)\nabla$ and the spectral edge is simple.
\end{thm}
\begin{rem}
	\leavevmode
	\begin{enumerate}
		\item While Theorem~\ref{theorem:2} achieves global simplicity for a Bloch eigenvalue, the perturbed operator is no longer a differential operator, i.e., it is non-local. In the theory of homogenization, non-local terms usually appear as limits of non-uniformly bounded operators~\cite{Briane2002},~\cite{BrianeCalc2002}. In the presence of crossing modes, non-locality appears in the theory of effective mass~\cite{Chabu2018}.
		\item Theorem~\ref{theorem:3} is an adaptation of the theorem of Klopp and Ralston~\cite{KloppRalston00} to divergence-type operators. Their proof relies heavily on the H\"older regularity for weak solutions of divergence-type operators. In our proof, we require H\"older continuity of the solutions as well as their derivatives. Hence, we have to impose $W^{1,\infty}$ condition on the coefficients.
		\item In Theorem~\ref{theorem:4}, we weaken the $W^{1,\infty}$ requirement on the coefficients under assumption of finiteness on the number of points at which the spectral edge is attained. This is essential for the applications that we have in mind, in the theory of homogenization, where only $L^\infty$ regularity is available on the coefficients.
	\end{enumerate}
\end{rem}

We shall also prove a theorem on internal edge homogenization, whose complete statement is deferred to Section~\ref{homogen}. Let $A\in M_B^>$. Birman and Suslina~\cite{Birman2006} propose an effective operator and prove operator error estimates with respect to the operator norm in $L^2(\mathbb{R}^d)$ for the limit as $\epsilon\to 0$ of the operator $\mathcal{A}^\epsilon\coloneqq-\nabla\cdot\left(A(\frac{x}{\epsilon})\nabla\right)$, at a non-zero spectral edge $\lambda_0$ under the regularity hypotheses $(A), (B), (C)$.

\begin{theorem*}
	Under appropriate modifications of the regularity hypotheses on the spectral edge, an effective operator is proposed as an approximation of the operator $\mathcal{A}^\epsilon$ in the limit $\epsilon\to 0$ at a multiple spectral edge, and operator error estimates with respect to the operator norm in $L^2(\mathbb{R}^d)$ are proved. 
\end{theorem*}

\begin{rem}
		Multiplicity of Bloch eigenvalues is a crucial difficulty in Bloch wave homogenization. The internal edge homogenization result is an attempt at circumventing this issue. Previously, this was handled by use of directional analyticity of Bloch eigenvalues for the linear elasticity operator whose lowest Bloch eigenvalue has multiplicty $3$~\cite{SivajiGanesh2005}.
\end{rem}

\begin{rem}
	\leavevmode
	\begin{enumerate}
		\item Bloch wave method belongs to the family of multiplier techniques in partial differential equations. In particular, exponential type multipliers, $e^{\tau\phi}$, with real exponents, are used in obtaining Carleman estimates for elliptic operators~\cite{rousseau12}.
		\item Any operator of the form $-\nabla\cdot A\nabla$ in $L^2(\mathbb{R}^d)$ may be written in direct integral form, provided $A$ is periodic. A satisfactory spectral theory for such operators is available for real symmetric $A$. However, non-selfadjoint operators are becoming increasingly important in physics~\cite{Sjoestrand09}. For non-symmetric $A$, the eigenvalues of the fibers $\mathcal{A}(\eta)$ may no longer be real and the eigenfunctions may not form a complete set. These difficulties were surmounted in proving the Bloch wave homogenization theorem for non-selfadjoint operators in~\cite{Sivaji2004}. Nevertheless, the generalized eigenfunctions form a complete set for a large class of elliptic operators of even order~\cite{agmon62}. However, we are not aware of physical interpretations of complex-valued Bloch-type eigenvalues.
		\item Most of the results of this paper would have similar analogues for internal edges of an elliptic system of equations, for example, the elasticity system. It would be interesting to consider these problems for the spectrum of non-elliptic operators such as the Maxwell operator. 
	\end{enumerate}
\end{rem}

The plan of this paper is as follows; in Section~\ref{local_simplicity}, we prove Theorem~\ref{theorem:1} on generic simplicity of Bloch eigenvalues at a point. In Section~\ref{global_simplicity}, we prove Theorem~\ref{theorem:2} and in subsequent sections~\ref{simplicity_edge_1} and~\ref{simplicity_edge_2}, we prove Theorems~\ref{theorem:3} and~\ref{theorem:4} concerning generic simplicity of spectral edges.  In the final section~\ref{homogen}, we give a short introduction to internal edge homogenization and furnish an application of perturbation theory to Bloch wave homogenization by proving Theorem~\ref{theorem:5}.

\section{Local Simplicity of Bloch eigenvalues}\label{local_simplicity}

Let $\eta_0\in Y^{'}$. Let $P$ be the set defined by \begin{equation*}P\coloneqq\{A\in M_B^{>}\mbox{  : the eigenvalues of $\mathcal{A}(\eta_0)$ are simple} \}.\end{equation*}

We can write the set $P$ as an intersection of countably many sets as follows: Let $P_0:=M_B^{>}$, and \begin{align*}P_n&:=\{A\in M_B^{>}: \mbox{ the first $n$ eigenvalues of $\mathcal{A}(\eta_0)$ } \mbox{are simple} \}.\\
&=\{A\in M_B^>: \lambda_1(\eta_0)<\ldots<\lambda_n(\eta_0)<\lambda_{n+1}(\eta_0)\leq\lambda_{n+2}(\eta_0)\leq\ldots\}.\end{align*}

Note that, \begin{align*}P \subseteq \ldots \subseteq P_n\subseteq P_{n-1}\subseteq\ldots\subseteq P_1\subseteq P_0
&\qquad\mbox{and}\qquad P=\bigcap_{n=0}^{\infty}P_n.\end{align*}
We shall require the following two lemmas.

\begin{lem}\label{lemma:11}
	$P_n$ is open in $M_B^{>}$ for all $n\in\mathbb{N}\cup\{0\}$.
\end{lem}

\begin{lem}\label{lemma:22}
	$P_{n+1}$ is dense in $P_n$, for all $n\in \mathbb{N}\cup\{0\}$.
\end{lem}

\begin{proof}(of Theorem \ref{theorem:1}) We recall that a property is said to be generic in a topological space $X$, if it holds on a set whose complement is of first category in $X$. We can write $P$ as the countable intersection $P=\displaystyle\bigcap_{n=0}^\infty P_n$, where $P_n$ is an open and dense set in $M_B^>$ for all $n\in\mathbb{N}\cup\{0\}$. Hence, the complement of $P$ is a set of first category. Therefore, the simplicity of eigenvalues of $\mathcal{A}(\eta_0)$ is a generic property in $M_B^>$.
\end{proof}

The rest of this section is devoted to the proofs of Lemmas~\ref{lemma:11} and~\ref{lemma:22}.

\subsection{Proof of Lemma~\ref{lemma:11}}

In this subsection, we begin by proving continuous dependence of the eigenvalues of the shifted operator $\mathcal{A}(\eta)$ on its coefficients. The main tool in this proof is Courant-Fischer min-max principle, which states that

\begin{align*}\lambda_m(\eta_0)=\min_{\dim F =m}\max_{v\in F}\frac{\int_{Y} A(\nabla+i\eta_0) v.\overline{(\nabla+i\eta_0) v}~dx}{\int_{Y} v^2~dx},\end{align*} where $F$ ranges over all subspaces of $H^1_\sharp(Y)$ of dimension $m$.

\begin{prop}
	Let $A_1,A_2\in M_B^{>}$ and let $\eta\mapsto\lambda_n^1(\eta), \eta\mapsto\lambda^2_n(\eta)$ be the n-th Bloch eigenvalues of the operators $\mathcal{A}_1$ and $\mathcal{A}_2$ respectively. Then \begin{align*}|\lambda_n^1(\eta_0)-\lambda^2_n(\eta_0)|\leq d c_n(\eta_0) ||A_1-A_2||_{L^\infty},\end{align*} where $c_n(\eta_0)$ is the $n^{th}$ eigenvalue of the shifted Laplacian $-(\nabla+i\eta_0)^2$ on $Y$ with periodic boundary conditions.
\end{prop}

\begin{proof}
	Let $a_1(v)=\int_{Y}A_1(\nabla+i\eta_0) v.\overline{(\nabla+i\eta_0) v}~dy$ and $a_2(v)=\int_{Y}A_2(\nabla+i\eta_0) v.\overline{(\nabla+i\eta_0) v}~dy$ be the quadratic forms that appear in the min-max principle.
	
	\begin{align*}
		|a_1(v)-a_2(v)|&=\left|\int_{Y}(A_1-A_2)(\nabla+i\eta_0) v.\overline{(\nabla+i\eta_0) v}~dy\right|\\
		&\leq d||A_1-A_2||_{L^\infty}\int_{Y}|(\nabla+i\eta_0) v|^2~dy,
	\end{align*}
	
	Therefore,
		\begin{align*}a_1(v)\leq a_2(v)+d||A_1-A_2||_{L^\infty}\int_{Y}|(\nabla+i\eta_0) v|^2~dy.	\end{align*}
	
	Now, divide both sides by $\int_Y |v|^2~dy$, the $L^2_\sharp(Y)$ inner product of $v$ with itself and apply the appropriate min-max to obtain \begin{align*}\lambda^1_m(\eta_0)\leq \lambda^2_m(\eta_0)+d c_m(\eta_0)||A_1-A_2||_{L^\infty}.\end{align*}
	
	Notice that the constant $c_m(\eta_0)$ is precisely the $m^{th}$ eigenvalue of the shifted Laplacian $-(\nabla+i\eta_0)^2$ on $Y$ with periodic boundary conditions. By interchanging the role of $A_1$ and $A_2$, the inequality \begin{align*}\lambda^2_m(\eta_0)\leq \lambda^1_m(\eta_0)+d c_m(\eta_0)||A_2-A_1||_{L^\infty},\end{align*} is obtained, which completes the proof of this proposition.
\end{proof}

\begin{rem}
	In~\cite{Conca1997}, the Bloch eigenvalues have been proved to be Lipschitz continuous in $\eta\in Y^{'}$. Indeed, one may prove that the Bloch eigenvalues are jointly continuous in $\eta\in Y^{'}$ and the coefficients of the operator.
\end{rem}

\begin{proof}[Proof of Lemma \ref{lemma:11}]
	Let $A\in P_n$ and 	\begin{align*}\delta=\min\{\lambda_{j+1}(\eta_0)-\lambda_{j}(\eta_0):j=1,2,\ldots,n\}.\end{align*}
	
	Let $\displaystyle c=\max_{1\leq j\leq n}d c_j(\eta_0)$, where $c_j(\eta_0)$ is the $j^{th}$ eigenvalue of the shifted Laplacian $-(\nabla+i\eta_0)^2$ on $Y$ with periodic boundary conditions.
	
	Let \begin{align*}U=\left\{A'\in M_B^>:||A-A'||_{L^\infty}< \frac{\delta}{4c}\right\}.\end{align*}
	
	$U$ is an open set in $M_B^>$ containing $A$. We shall show that $U$ is a subset of $P_n$. Let $A'\in 
	U$. Let $\{\lambda'_j(\eta),~j=1,2,\ldots\}$ be the Bloch eigenvalues of operator $\mathcal{A}'$ associated to $A'$. For $j=1,2,\ldots,n$, we have:
	\begin{align*}
		|\lambda'_j(\eta_0)-\lambda_j(\eta_0)| & \leq dc_j(\eta_0)||A-A'||_{L^\infty} \leq  dc_j(\eta_0)\frac{\delta}{4c}\leq\frac{\delta}{4}.
	\end{align*}
	
	Hence, \begin{align*}
		\delta & \leq  \lambda_{j+1}(\eta_0)-\lambda_j(\eta_0)\\
	           & \leq  |\lambda'_{j+1}(\eta_0)-\lambda_{j+1}(\eta_0)|+|\lambda'_j(\eta_0)-\lambda'_{j+1}(\eta_0)|+|\lambda'_j(\eta_0)-\lambda_j(\eta_0)|\\
		      & \leq \frac{\delta}{4}+ |\lambda'_j(\eta_0)-\lambda'_{j+1}(\eta_0)|+\frac{\delta}{4}\\
		      & = \frac{\delta}{2}+\lambda'_{j+1}(\eta_0)-\lambda'_j(\eta_0).
	\end{align*}
	
	Therefore, $\lambda'_{j+1}(\eta_0)-\lambda'_j(\eta_0)\geq\frac{\delta}{2}>0$ for $j=1,2,\ldots,n.$
	Therefore, the first $n$ Bloch eigenvalues of $\mathcal{A}'$ are simple at $\eta_0$, as required.
\end{proof}

\subsection{Proof of Lemma~\ref{lemma:22}}
In this section, we shall use perturbation theory of selfadjoint operators to prove 
Lemma~\ref{lemma:22}. Let $A\in M_B^>$ and $B$ be a symmetric matrix with $L^\infty_\sharp(Y,\mathbb{R})$-entries. For $|\tau|<\sigma_0\coloneqq\frac{\alpha}{2d||B||_{L^\infty}}$, $A+\tau B\in M^>_B$, where $\alpha$ is a coercivity constant for $A$ as in~\eqref{coercivity}. Consider the operator $\mathcal{A}(\eta_0)+\tau\mathcal{B}(\eta_0)$ in $L^2_\sharp(Y)$. We shall prove in Appendix~\ref{PerturbationTheory}, that the operator family $\mathcal{F}(\tau)=\mathcal{A}(\eta_0)+\tau\mathcal{B}(\eta_0)$ is a selfadjoint holomorphic family of type $(B)$ for $|\tau|<\sigma_0$. For its definition and related notions, see Kato~\cite{Kato1995}.

We shall make use of the following theorem which asserts the existence of a sequence of eigenpairs associated with a selfadjoint holomorphic family of type $(B)$, analytic in $\tau\in(-\sigma_0,\sigma_0)$. The proof of this theorem dates back to Rellich, hence we shall call these eigenvalue branches as Rellich branches.

\begin{thm}(Kato-Rellich)\label{katorellich} Let $\mathcal{A}(\eta_0)(\tau)$ be a selfadjoint holomorphic family of type $(B)$, defined for $\tau\in R$ where $R=\{z\in\mathbb{C}:|\operatorname{Re}({z})|<\sigma_0,|\operatorname{Im}({z})|<\sigma_0\}$ and $\sigma_0\coloneqq \frac{\alpha}{2d||B||_{L^\infty}}$. Let $\mathcal{A}(\eta_0)(\tau)+C_*I$ have compact resolvent for some $C_*\in \mathbb{R}$. Then, there exists a sequence of scalar-valued functions $(\lambda_j(\tau;\eta_0))_{j=1}^\infty$ and $L^2_\sharp(Y)$-valued functions $(u_j(\tau;\eta_0))_{j=1}^\infty$ defined on $I=(-\sigma_0,\sigma_0)$, such that
	\begin{enumerate}
		\item For each fixed $\tau\in I$, the sequence $(\lambda_j(\tau;\eta_0))_{j=1}^\infty$ represents all the eigenvalues of $\mathcal{A}(\eta_0)(\tau)$ counting multiplicities and the functions $(u_j(\tau;\eta_0))_{j=1}^\infty$ represent the corresponding eigenvectors.
		\item For each $j\in\mathbb{N}$, the functions $(\lambda_j(\tau;\eta_0))_{j=1}^\infty$ and $(u_j(\tau;\eta_0))_{j=1}^\infty$ are analytic on $I$ with values in $\mathbb{R}$ and $L^2_\sharp(Y)$ respectively.
		\item The sequence $(u_j(\tau;\eta_0))_{j=1}^\infty$ is orthonormal in $L^2_\sharp(Y).$
		\item Suppose that the $m^{th}$ eigenvalue of $\mathcal{A}(\eta_0)(\tau)$ at $\tau=0$ has multiplicity $p$, i.e., \begin{align*}\lambda_m(0;\eta_0)=\lambda_{m+1}(0;\eta_0)=\ldots=\lambda_{p+m-2}(0;\eta_0)=\lambda_{p+m-1}(0;\eta_0).\end{align*} For each interval $K\subset\mathbb{R}$ with $\overline{K}$ containing the eigenvalue $\lambda_m(0;\eta_0)$ and no other eigenvalue, $\lambda_m(\tau;\eta_0), $ $\lambda_{m+1}(\tau;\eta_0),\ldots, \lambda_{p+m-1}(\tau;\eta_0)$ are the only eigenvalues of $\mathcal{A}(\eta_0)(\tau)$, counting multiplicities, lying in the interval $K$.
	\end{enumerate}
\end{thm}

 By Kato-Rellich Theorem, an eigenvalue $\lambda(\eta_0)$ of $\mathcal{F}(0)$ of multiplicity $h$, splits into $h$ analytic functions $(\lambda_m(\tau;\eta_0))_{m=1}^h$. Further, the corresponding eigenfunctions $(u_m(\tau;\eta_0))_{m=1}^h$ are also analytic. Let the $h$ eigenvalues and eigenvectors of $\mathcal{F}(\tau)$ have the following power series expansions at $\tau=0$ for $m=1,2,\ldots,h$:
\begin{align*}
\lambda_m(\tau;\eta_0)=\lambda(\eta_0)+\tau a_m(\eta_0)+\tau^2\beta_m(\tau,\eta_0)\notag\\
u_{m}(\tau;\eta_0)=u_m(\eta_0)+\tau v_m(\eta_0)+\tau^2 w_m(\tau,\eta_0).
\end{align*}

The proof of Lemma~\ref{lemma:22} will rely on the fact that we may choose $B$ in such a way that $a_m(\eta_0)\neq a_n(\eta_0)$ for some $m,n\in\{1,2,\ldots,h\}$. Then, for sufficiently small $\tau$, $\lambda_m(\tau;\eta_0)\neq\lambda_n(\tau;\eta_0)$. In that case, the multiplicity of the perturbed Bloch eigenvalue at $\eta_0$ will be less than $h$.

The eigenpairs satisfy the following equation:
\begin{align*}
\left(-(\nabla+i\eta_0)\cdot (A+\tau B)(\nabla+i\eta_0)-\lambda_m(\tau,\eta_0)\right)u_m(\tau,\eta_0)=0.
\end{align*}

Differentiating the above with respect to $\tau$ and setting $\tau$ to $0$, we obtain:
\begin{align*}
	-(\nabla+i\eta_0)\cdot A(\nabla+i\eta_0)v_m(\eta_0)-(\nabla+i\eta_0)\cdot B(\nabla+i\eta_0)u_m(\eta_0)-\lambda(\eta_0)v_m(\eta_0)\notag-a_m(\eta_0)u_m(\eta_0)=0
\end{align*}

Finally, multiply by $u_n(\eta_0)$ and integrate over $Y$ to conclude that
\begin{align}\label{hellmanfeynman}
	\int_Y B(\nabla+i\eta_0)u_m(\eta_0)\cdot(\nabla-i\eta_0)\overline{u_n(\eta_0)}~dy=a_m(\eta_0)\delta_{mn}.
\end{align}


Equation~\eqref{hellmanfeynman} suggests the following construction. Given a perturbation $B$ and a basis $F=\{f_1,f_2,\ldots,f_h\}$ for the unperturbed eigenspace $N(\eta_0)\coloneqq ker(\mathcal{A}(\eta_0)-\lambda(\eta_0)I)$,  we can define a selfadjoint operator $G_B$ on $N(\eta_0)$ whose matrix in the basis $F$ is given by 
\begin{align*}\left([G_B]_F\right)_{m,n}\coloneqq\int_Y B(\nabla+i\eta_0)f_m\cdot(\nabla-i\eta_0)\overline{f_n}~dy.\end{align*}

In particular, it follows from equation~\eqref{hellmanfeynman} that in the basis of unperturbed eigenfunctions $E=\{u_1(\eta_0), u_2(\eta_0),\ldots,u_h(\eta_0)\}$, $[G_B]_E$ is a diagonal matrix, 
\begin{align*}[G_B]_E=diag(a_1(\eta_0),a_2(\eta_0),\ldots,a_h(\eta_0)).\end{align*} If $[G_B]_E$ is a scalar matrix, then the operator $G_B$ is a scalar multiple of identity operator. However, if we can find a basis $F$ for the eigenspace and a matrix $B$, corresponding to which, the matrix $[G_B]_F$ has a non-zero off-diagonal entry, then for that choice of $B$, $[G_B]_E$ will not be a scalar matrix, and hence, $a_m(\eta_0)\neq a_n(\eta_0)$ for some $m,n\in\{1,2,\ldots,h\}$.

\begin{prop}\label{find B}
	There exists a symmetric matrix $B$ with $L^\infty_\sharp(Y,\mathbb{R})$-entries such that the operator $G_B$ is not a scalar multiple of identity.	
\end{prop}

\begin{proof}
	As noted earlier, the proposition will be proved if we can find a basis $F$ and a matrix $B$ with $L^\infty_\sharp(Y,\mathbb{R})$ entries, such that the matrix $[G_B]_F$ has a non-zero off-diagonal entry.
	
	Let $F=\{f_1,f_2,\ldots,f_h\}$ be any basis of $ker(\mathcal{A}(\eta_0)-\lambda(\eta_0)I)$. Suppose that for some $j\in\{1,2,\ldots,d\}$, 
	\begin{align}\label{alternative1}
	(\partial_j+i\eta_{0,j})f_1(\partial_j-i\eta_{0,j})\overline{f_2}\not\equiv 0,
	\end{align}
	where $\eta_0=(\eta_{0,1},\eta_{0,2},\ldots,\eta_{0,d})$. Since, $f_i\in H_\sharp^1(Y)$, $g\coloneqq(\partial_j+i\eta_{0,j})f_1(\partial_j-i\eta_{0,j})\overline{f_2}\in L^1_\sharp(Y)$. Hence, by Hahn-Banach Theorem, there is a continuous linear functional $\kappa\in(L^1_\sharp(Y))^*$, such that $\kappa(g)=||g||\neq 0.$ However, by duality, there exists a $\beta\in L^\infty_\sharp(Y)$, such that $\kappa(g)=\int_Y \beta\,gdy=||g||\neq 0.$
	
	Now, either $\int_Y \operatorname{Re}(\beta)g\neq 0$ or $\int_Y \operatorname{Im}(\beta)g\neq 0$. Suppose, without loss of generality that $\int_Y \operatorname{Re}(\beta)g\neq 0$ and define 	\begin{align*}B=diag(0,0,\ldots,0,\operatorname{Re}(\beta),0,\ldots,0)\end{align*} with $\operatorname{Re}(\beta)$ in the $j^{th}$ place, then
	\begin{align*}
		([G_B]_F)_{1,2}&=\int_Y B(\nabla+i\eta_0)f_1\cdot(\nabla-i\eta_0)\overline{f_2}~dy\\
		&=\int_Y\operatorname{Re}(\beta)(\partial_j+i\eta_{0,j})f_1(\partial_j-i\eta_{0,j})\overline{f_2}~dy\\
		&=\int_Y\operatorname{Re}(\beta) g~dy\neq 0.
	\end{align*}
	Alternatively, if $(\nabla+i\eta_0)f_1\cdot(\nabla-i\eta_0)\overline{f_2}\equiv 0$, then there exists $j\in\{1,2,\ldots,d\}$, such that
	\begin{align}\label{alternative2}
	|(\partial_j+i\eta_{0,j})f_1|^2-|(\partial_j+i\eta_{0,j})f_2|^2\not\equiv 0.
	\end{align}
	It is easy to see that if~\eqref{alternative1} and~\eqref{alternative2} do not hold, then $f_1$ and $f_2$ are both a scalar multiple of $\exp(i\eta_0\cdot y)$, which contradicts the fact that they are distinct elements of basis of $N=ker(\mathcal{A}(\eta_0)-\lambda(\eta_0)I)$.
	
	Since, for all $m=1,2,\ldots,h$, $f_m\in H_\sharp^1(Y)$, $g{'}\coloneqq|(\partial_j+i\eta_{0,j})f_1|^2-|(\partial_j+i\eta_{0,j})f_2|^2\in L^1_\sharp(Y,\mathbb{R})$. Hence, by Hahn-Banach Theorem, there is a continuous linear functional $\kappa{'}\in(L^1_\sharp(Y,\mathbb{R}))^*$, such that $\kappa{'}(g{'})=||g{'}||\neq 0.$ However, by duality, there exists a $\beta{'}\in L^\infty_\sharp(Y,\mathbb{R})$, such that $\kappa{'}(g{'})=\int_Y \beta{'}g{'}=||g{'}||\neq 0.$
	
	Define 	\begin{align*}B=diag(0,0,\ldots,0,\beta{'},0,\ldots,0)\end{align*} with $\beta{'}$ in the $j^{th}$ place, then in the new basis $F^{'}=\{f_1+f_2,f_1-f_2,f_3,\ldots,f_h\}$, the $(1,2)^{th}$ entry of $[G_B]_{F^{'}}$ is given by 
	\begin{align*}
		\int_Y B(\nabla+i\eta_0)(f_1+f_2)\cdot(\nabla-i\eta_0)(\overline{f_1}-\overline{f_2})~dy =	\int_Y \beta{'}|(\partial_j+i\eta_{0,j})f_1|^2-|(\partial_j+i\eta_{0,j})f_2|^2~dy\neq 0.
	\end{align*}	
	Thus, either way, we have found a basis in which an off-diagonal entry of $[G_B]_F$ is non-zero. Hence, the operator $G_B$ is not a scalar multiple of identity. In particular, the matrix $[G_B]_E$ cannot be a scalar matrix.
\end{proof}

\begin{proof}[Proof of Lemma \ref{lemma:22}]
	Let $A\in P_n$. Given $\epsilon>0$, we want to find $A'\in P_{n+1}$ such that $||A-A'||_{L^\infty}<\epsilon$. We shall construct $A'$ in the form $A'=A+\tau B$, where $B$ is a symmetric matrix with $L^\infty_\sharp(Y,\mathbb{R})$-entries and $\tau\in\mathbb{R}$. By Lemma~\ref{lemma:11}, we can choose $\tau_0$ so that $A+\tau B\in P_n$ for $|\tau|<\tau_0$. Hence, the first $n$ eigenvalues of the operator $-(\nabla+i\eta_0)\cdot(A+\tau B)(\nabla+i\eta_0)$ are simple for $|\tau|<\tau_0$. Subsequently, we must choose $\tau$ such that $|\tau|<\sigma_0=\frac{\alpha}{2d||B||_{L^\infty}}$, in order to apply the Kato-Rellich Theorem. Now, suppose that the $(n+1)^{th}$ eigenvalue of $\mathcal{A}(\eta_0)$ has multiplicity $h$. By Kato-Rellich Theorem (Theorem~\ref{katorellich}), the $h$ eigenvalue branches of the perturbed operator $\mathcal{A}(\eta_0)+\tau\mathcal{B}(\eta_0)$ are given by the following power series at $\tau=0$, for $r=1,2,\ldots,h$:
	\begin{align*}
	\lambda_r(\tau;\eta_0)=\lambda(\eta_0)+\tau a_r(\eta_0)+\tau^2\beta_r(\tau;\eta_0).
	\end{align*}
	If there are $m,n\in\{1,2,\ldots,h\}$ such that $a_m(\eta_0)\neq a_n(\eta_0)$, then there is a $\tau_1$ such that, $\lambda_m(\tau;\eta_0)\neq\lambda_n(\tau;\eta_0)$ for $|\tau|<\tau_1$. Since two of the $h$ eigenvalue branches are distinct for small $\tau$, the multiplicity of the perturbed eigenvalue, which can only go down for small $\tau$, must be less than or equal to $h-1$. This can be achieved through an application of Proposition~\ref{find B} which gives us a matrix $B_1$ such that at least two of $(a_r(\eta_0))_{r=1}^h$ are distinct. Now, starting from the matrix $A+\tau_1 B_1$, we repeat the procedure above so that the multiplicity of the $(n+1)^{th}$ eigenvalue is further reduced. The perturbed matrix is now labelled $A+\tau_1B_1+\tau_2B_2$.  Finally, after a finite number of such steps, we can reduce the multiplicity of the $(n+1)^{th}$ eigenvalue to $1$. At the end of this procedure, we obtain a matrix of the form $A'=A+\sum_{j=1}^N\tau_jB_j$, for some $N\in\mathbb{N}$. Each perturbation must be chosen so that $\sum_{j=1}^{N}\tau_j||B_j||_{L^\infty}<\epsilon$.
\end{proof}

\begin{rem}\label{simple_in_neighborhood}
	Theorem~\ref{theorem:1} proves that an eigenvalue $\lambda(\eta_0)$ of the shifted operator $\mathcal{A}(\eta_0)$ can be made simple by a perturbation of the matrix $A\in M_B^>$. However, since the Bloch eigenvalues are Lipschitz continuous functions of the parameter $\eta\in Y^{'}$~\cite{Conca1997}, the perturbed eigenvalue $\tilde{\lambda}(\eta)$ will continue to remain simple in some neighborhood of $\eta_0$.
\end{rem}

\begin{rem}
	\leavevmode
	\begin{enumerate}
		\item The perturbation formula~\eqref{hellmanfeynman} may be thought of as a variation of the Hellmann-Feynman theorem in the physics literature. The coefficients of the differential operator~\eqref{eq1:operator} are real-valued functions, in as much as they are related to properties of materials. The presence of complex-valued coefficients in the perturbation formula complicates the choice of the real-valued perturbation $B$.
		\item In the theory of homogenization, the coefficients of the second order divergence-type periodic elliptic operator are usually only measurable and bounded. By regularity theory~\cite{Ladyzhenskaya68}, the eigenfunctions of the shifted operator $\mathcal{A}(\eta)$ are known to be H\"older continuous. However, derivatives of eigenfunctions, which may not be bounded, appear in the perturbation formula~\eqref{hellmanfeynman}. Therefore, the perturbation $B$ is chosen using the Hahn-Banach Theorem.
	\end{enumerate}
\end{rem}

\section{Global Simplicity}\label{global_simplicity}

In the previous section, we have proved that a given Bloch eigenvalue $\lambda_m(\eta)$ of the operator $\mathcal{A}$ can be made simple locally in $Y^{'}$ through a small perturbation in the coefficients. In this section, we shall perform perturbation on the operator $\mathcal{A}$ in such a way that its spectrum still retains the fibered character, i.e., $\sigma(\tilde{A})=\cup_{\eta\in Y^{'}}\sigma(\tilde{A}(\eta))$ and the $m^{th}$ eigenvalue function $\eta\mapsto\tilde{\lambda}_m(\eta)$ is simple for all $\eta\in Y^{'}$. However, the perturbed operator $\tilde{\mathcal{A}}$ may no longer be a differential operator.

\begin{proof}[Proof of Theorem~\ref{theorem:2}] 
	The operator~\eqref{eq1:operator} has a direct integral decomposition $\mathcal{A}=\int_{\eta\in\mathbb{T}^d}^{\bigoplus}\mathcal{A}(\eta)d\eta$ where $\mathcal{A}(\eta)=-(\nabla+i\eta)\cdot A(\nabla+i\eta)$ is an unbounded operator in $L^2_\sharp(Y)$. We would like to point out that $Y^{'}$ is understood to parametrize the torus, $\mathbb{T}^d$. Consider the $m^{th}$ Bloch eigenvalue $\lambda_m(\eta)$ of $\mathcal{A}$. By Lemma~\ref{lemma:22}, at any point $\eta_0\in Y^{'}$, we can find a perturbation of the coefficients $A=(a_{kl})$ of $\mathcal{A}(\eta_0)$ so that the perturbed eigenvalue $\tilde{\lambda}_m(\eta_0)$ is simple. By Remark~\ref{simple_in_neighborhood}, there is a neighborhood of $\eta_0$, $\mathcal{G}_{\eta_0}$ in which the perturbed eigenvalue $\tilde{\lambda}_m(\eta)$ of the perturbed shifted operator $\tilde{\mathcal{A}}(\eta)$ is simple. In this manner, for each $\xi\in \mathbb{T}^d$, we obtain a perturbation $B_{\xi}$ and a neighborhood, $\mathcal{G}_{\xi}$ in which the eigenvalue of the perturbed operator $\tilde{\mathcal{A}}(\eta)=-(\nabla+i\eta)\cdot (A+B_{\xi})(\nabla+i\eta)$ is simple. These sets form an open cover of the torus. By compactness of $\mathbb{T}^d$, there is a finite subcover having the property that in each member $\mathcal{G}_{\xi}$ of the subcover, the corresponding perturbation $B_{\xi}$ causes the perturbed eigenvalue $\tilde{{\lambda}}_m(\eta)$ to be simple in $\mathcal{G}_{\xi}$.
	
	Let $\{\mathcal{G}_1,\mathcal{G}_2,\ldots,\mathcal{G}_n\}$ be the finite subcover of the torus obtained above. Define $\mathcal{O}_1=\mathcal{G}_1$. For $r\geq 1$, define $\mathcal{O}_{r+1}=\displaystyle\mathcal{G}_{r+1}\setminus\bigcup_{j=1}^r\mathcal{G}_j$. Suppose that $B_j$ is the perturbation corresponding to the set $\mathcal{O}_j$.
	
	Now, define the parametrized operator 	\begin{align*}\tilde{\mathcal{A}}(\eta)=-(\nabla+i\eta)\cdot (A+\sum_{j=1}^n B_{j}\,\chi_{\mathcal{O}_j})(\nabla+i\eta)\end{align*}
	which depends measurably on $\eta\in\mathbb{T}^d$. Finally, define the direct integral $\tilde{\mathcal{A}}=\int_{\eta\in\mathbb{T}^d}^{\bigoplus}\tilde{\mathcal{A}}(\eta)d\eta$, where each of the fibers is a differential operator in $L^2_\sharp(Y)$. Then, it is known~\cite[p.284]{Reed1978} that, 
	\begin{equation*}
	\sigma(\tilde{A})=\bigcup_{\eta\in Y^{'}}\sigma(\tilde{A}(\eta)).
	\end{equation*}
	
	Hence, we may define an $m^{th}$ eigenvalue function $\eta\mapsto\tilde{\lambda}_m(\eta)$ with the property that \begin{align*}|\lambda_m(\eta)-\tilde{\lambda}_m(\eta)|\leq C\max_{1\leq j\leq n}||B_j||_{L^{\infty}},\end{align*} where $\lambda_m(\eta)$ is the $m^{th}$ Bloch eigenvalue of $\mathcal{A}$.
\end{proof}

\begin{rem}
	\leavevmode
	\begin{enumerate}
		\item Although the $m^{th}$ eigenvalue of the perturbed operator is simple for all parameter values, $\tilde{\lambda}_m(\eta)$ may only be measurable in $\eta\in Y^{'}$. However, $\tilde{\lambda}_m(\eta)$ is analytic in each $\mathcal{O}_j\subset \mathbb{T}^d$.
		\item The perturbed operator $\tilde{\mathcal{A}}$ is no longer a differential operator, even though each fiber $\tilde{\mathcal{A}}(\eta)$ is a differential operator. In fact. we shall prove in Theorem~\ref{notdiffop} that $\tilde{\mathcal{A}}$ is a differential operator if and only if $B_1=B_2=\ldots=B_n.$
		\item A rigorous account of direct integral decomposition of operators, such as the one employed above for periodic operators, may be found in~\cite{schmudgen90} and~\cite{maurin68}.
	\end{enumerate}
\end{rem}
\begin{lem}\label{notdifferential}
	Let $B$ be a symmetric matrix with $L^\infty_\sharp(Y,\mathbb{R})$-entries. Define $\mathcal{B}(\eta)=-(\nabla+i\eta)\cdot B(\nabla+i\eta)$. Let $\mathcal{O}\subset Y^{'}$ be a proper subset of $Y^{'}$. Then, the direct integral defined by $\mathcal{B}=\int^{\bigoplus}_{\eta\in\mathbb{T}^d}\mathcal{B}(\eta)\chi_{\mathcal{O}}$ is not a differential operator.
\end{lem}

\begin{proof}
	By Peetre's Theorem~\cite{peetre60},~\cite[p.~236]{duistermaatkolk2010}, a linear operator $\mathcal{B}:\mathcal{D}(\mathbb{R}^d)\to\mathcal{D}^{'}(\mathbb{R}^d)$ is a differential operator if and only if $supp(Pu)\subset supp(u)$ for all $u\in\mathcal{D}(\mathbb{R}^d)$. Here, $\mathcal{D}(\mathbb{R}^d)$ denotes the space of compactly supported smooth functions on $\mathbb{R}^d$ with the topology of test functions. Also, let $\mathcal{S}(\mathbb{R}^d)$ denote the Schwartz class of rapidly decreasing smooth functions on $\mathbb{R}^d$. In order to show that $\mathcal{B}$ is not a differential operator, we will show that it does not preserve supports. 
	
	Given $g\in\mathcal{D}(\mathbb{R}^d)$, we define its Gelfand transform as 	\begin{align*}g_\sharp(y,\eta)=\sum_{p\in\mathbb{Z}^d}g(y+2\pi p)e^{-i(y+2\pi p)\cdot\eta}.\end{align*} This is a function in $L^2(Y^{'},L^2_\sharp(Y))$. The map from $g\mapsto g_\sharp$ is an isometry on $\mathcal{D}(\mathbb{R}^d)$ in the $L^2$-inner product and hence it may be extended to a unitary isomorphism from $L^2(\mathbb{R}^d)$ to $L^2(Y^{'},L^2_\sharp(Y))$. We shall show that $\mathcal{B}(g)$ is not compactly supported. $\mathcal{B}(g)$ is a tempered distribution defined as:
	\begin{equation*}
	(\mathcal{B}(g),\phi)=\int_{\mathcal{O}}\int_Y B(\nabla+i\eta)g_\sharp(y,\eta)\cdot(\nabla-i\eta)\overline{\phi}_\sharp(y,\eta)~dyd\eta.
	\end{equation*}
	
	We may define the Fourier transform of $\mathcal{B}(g)$ in $\mathcal{S}^{'}(\mathbb{R}^d)$ as \begin{align*}(\widehat{\mathcal{B}(g)},\phi)=(\mathcal{B}(g),\mathcal{F}^{-1}(\phi)),\end{align*}
	where $\mathcal{F}^{-1}(\phi)=\frac{1}{(2\pi)^{d/2}}\int_{\mathbb{R}^d}\phi(\eta)e^{iy\cdot\eta}~d\eta$ is the inverse Fourier transform of $\phi$. Since $\phi\in\mathcal{S}(\mathbb{R}^d)$, there exists a $\psi\in\mathcal{S}(\mathbb{R}^d)$ such that $\phi=\widehat{\psi}$. Therefore, 	\begin{align*}(\widehat{\mathcal{B}(g)},\phi)=(\mathcal{B}(g),\mathcal{F}^{-1}(\phi))=(\mathcal{B}(g),\psi).\end{align*}
	By Poisson Summation Formula~\cite[p.~171]{grafakosclassical2008}, we conclude that
	\begin{align}\label{poissonsum}
		\psi_\sharp(y,\eta)&=\sum_{p\in\mathbb{Z}^d}\psi(y+2\pi p)e^{-i(y+2\pi p)\cdot\eta}=\frac{1}{(2\pi)^{d/2}}\sum_{q\in\mathbb{Z}^d}\widehat{\psi}(\eta+q)e^{iq\cdot y}\notag\\&\quad=\frac{1}{(2\pi)^{d/2}}\sum_{q\in\mathbb{Z}^d}{\phi}(\eta+q)e^{iq\cdot y}.
	\end{align}
	
	Now, suppose that $\phi\in\mathcal{S}(\mathbb{R}^d)$ vanishes on $\bigcup_{q\in\mathbb{Z}^d}(\mathcal{O}+q)$, then $\psi_\sharp$, as obtained in~\eqref{poissonsum}, vanishes on $\mathcal{O}$. Hence, 
	\begin{align*}
		(\widehat{\mathcal{B}(g)},\phi)&=(\mathcal{B}(g),\psi)=\int_Y\int_{\mathcal{O}}B(\nabla+i\eta)g_\sharp(y,\eta)\cdot(\nabla-i\eta)\overline{\psi}_\sharp(y,\eta)~d\eta\,dy=0.
	\end{align*}
	Therefore, $\widehat{\mathcal{B}(g)}$ vanishes on the open set $\bigcup_{q\in\mathbb{Z}^d}(\mathcal{O}+q)$. By Schwartz-Paley-Wiener Theorem~\cite[p.~191]{rudinfunctional1991}, $\widehat{\mathcal{B}(g)}$ cannot be the Fourier transform of a compactly supported distribution, i.e., $\mathcal{B}(g)$ is not compactly supported. 
\end{proof}

\begin{thm}
	Let $\{\mathcal{O}_1,\mathcal{O}_2,\ldots,\mathcal{O}_n\}$ be a partition of $Y^{'}$ up to a set of measure zero, i.e., $Y^{'}\setminus\bigcup_{j=1}^n\mathcal{O}_j$ is a set of measure zero. Define $\mathcal{B}:\mathcal{D}(\mathbb{R}^d)\to\mathcal{D}^{'}(\mathbb{R}^d)$ by $\mathcal{B}(g)=\sum_{j=1}^n\int^{\bigoplus}_{\eta\in\mathcal{O}_j}\mathcal{B}_j(\eta)g_\sharp(y,\eta)$ where $\mathcal{B}_j(\eta)=-(\nabla+i\eta)\cdot B_j(\nabla+i\eta)$ where for all $j\in\{1,2,\ldots,n\}$, $B_j$ are matrices with $L^\infty_\sharp(Y,\mathbb{R})$-entries, then $\mathcal{B}$ is a differential operator if and only if $B_1=B_2=\ldots=B_n$.\label{notdiffop}
\end{thm}

\begin{proof}
	If $B\coloneqq B_1=B_2=\ldots=B_n$, then $\mathcal{B}(g)=-\nabla\cdot B\nabla(g)$ which is a differential operator.
	
	Conversely, without loss of generality, assume that $B_1\neq B_2$ and suppose that $\mathcal{B}$ is a differential operator. Then,
	\begin{align*}
		\mathcal{B}(g)=&\int^{\bigoplus}_{\eta\in Y^{'}}\mathcal{B}_1(\eta)d\eta+\int^{\bigoplus}_{\eta\in\mathcal{O}_2}(\mathcal{B}_2-\mathcal{B}_1)(\eta)d\eta+\int^{\bigoplus}_{\eta\in\mathcal{O}_3}(\mathcal{B}_3-\mathcal{B}_1)(\eta)d\eta+\ldots\\&\qquad
		+\int^{\bigoplus}_{\eta\in\mathcal{O}_n}(\mathcal{B}_n-\mathcal{B}_1)(\eta)d\eta.
	\end{align*}
	Hence,
	\begin{align*}
		\mathcal{B}(g)-\int^{\bigoplus}_{\eta\in Y^{'}}\mathcal{B}_1(\eta)d\eta=\sum_{j=2}^{n}\int^{\bigoplus}_{\eta\in\mathcal{O}_j}(\mathcal{B}_j-\mathcal{B}_1)(\eta)d\eta
	\end{align*}
	The left hand side of the above equation is a differential operator. We will show that the right hand side is not a differential operator to obtain a contradiction.
	
	We proceed as in Lemma~\ref{notdifferential}.
	
	Define $\mathcal{C}:\mathcal{D}(\mathbb{R}^d)\to\mathcal{D}^{'}(\mathbb{R}^d)$ by 
	\begin{equation*}
	(\mathcal{C}(g),\phi)=\sum_{j=2}^{n}\int_{\mathcal{O}_j}\int_Y (B_j-B_1)(\nabla+i\eta)g_\sharp(y,\eta)\cdot(\nabla-i\eta)\overline{\phi}_\sharp(y,\eta)~dy\,d\eta
	\end{equation*}
	
	It is easy to see that $\mathcal{C}(g)\in\mathcal{S}^{'}(\mathbb{R}^d)$. 
	
	Therefore, we may define its Fourier transform by
		\begin{align*}(\widehat{\mathcal{C}(g)},\phi)=(\mathcal{C}(g),\mathcal{F}^{-1}(\phi)),\end{align*}
	where $\mathcal{F}^{-1}(\phi)=\frac{1}{(2\pi)^{d/2}}\int_{\mathbb{R}^d}\phi(\eta)e^{iy\cdot\eta}~d\eta$ is the inverse Fourier transform of $\phi$. Since $\phi\in\mathcal{S}(\mathbb{R}^d)$, there exists $\psi\in\mathcal{S}(\mathbb{R}^d)$ such that $\phi=\widehat{\psi}$. Therefore, 	\begin{align*}(\widehat{\mathcal{C}(g)},\phi)=(\mathcal{C}(g),\mathcal{F}^{-1}(\phi))=(\mathcal{C}(g),\psi).\end{align*}
	By Poisson Summation Formula~\cite[p.~171]{grafakosclassical2008}, we conclude that
	\begin{align}\label{poissonsum2}
		\psi_\sharp(y,\eta)&=\sum_{p\in\mathbb{Z}^d}\psi(y+2\pi p)e^{-i(y+2\pi p)\cdot\eta}= \frac{1}{(2\pi)^{d/2}}\sum_{q\in\mathbb{Z}^d}\widehat{\psi}(\eta+q)e^{iq\cdot y}\notag\\
		&=\frac{1}{(2\pi)^{d/2}}\sum_{q\in\mathbb{Z}^d}{\phi}(\eta+q)e^{iq\cdot y}.
	\end{align}
	
	Now, suppose that $\phi\in\mathcal{S}(\mathbb{R}^d)$ vanishes on $\bigcup_{q\in\mathbb{Z}^d}(\bigcup_{j=2}^n\mathcal{O}_j+q)$, then $\psi_\sharp$, as obtained in~\eqref{poissonsum2}, vanishes on $\bigcup_{j=2}^n\mathcal{O}_j$. Hence, 
	\begin{align*}
		(\widehat{\mathcal{C}(g)},\phi)&=(\mathcal{C}(g),\psi)\\
		&=\sum_{j=2}^{n}\int_Y\int_{\mathcal{O}_j}(B_j-B_1)(\nabla+i\eta)g_\sharp(y,\eta)\cdot(\nabla-i\eta)\overline{\psi}_\sharp(y,\eta)~d\eta\,dy=0.
	\end{align*}
	Therefore, $\widehat{\mathcal{C}(g)}$ vanishes on the open set $\bigcup_{q\in\mathbb{Z}^d}(\bigcup_{j=2}^n\mathcal{O}_j+q)$. By Schwartz-Paley-Wiener Theorem~\cite[p.~191]{rudinfunctional1991}, $\widehat{\mathcal{C}(g)}$ cannot be the Fourier transform of a compactly supported distribution, i.e., $\mathcal{C}(g)$ is not compactly supported. Therefore, $\mathcal{C}$ is not a differential operator.
\end{proof}

\section{Proof of Theorem~\ref{theorem:3}}\label{simplicity_edge_1}
In this section, we prove that a spectral edge of a periodic elliptic differential operator can be made simple through a perturbation in the coefficients. The proof essentially follows Klopp and Ralston~\cite{KloppRalston00}, with the straightforward modification that the coefficients must come from $W^{1,\infty}_\sharp(Y,\mathbb{R})$. This condition is required to ensure that the eigenfunctions and their derivatives are H\"older continuous functions. We produce the proof here for completeness.

Suppose that the coefficients of the operator~\eqref{eq1:operator}, $a_{kl}\in W^{1,\infty}_\sharp(Y)$. Note that the Bloch eigenvalues which are defined for $\eta\in Y^{'}$ are Lipschitz continuous in $\eta$ and may be extended as periodic functions to $\mathbb{R}^d$. In the sequel, we shall treat the Bloch eigenvalues as functions on $\mathbb{T}^d$, which is identified with $Y^{'}$ in a standard way. Also, we shall write $\lambda_j(\eta,A)$ to specify that a Bloch eigenvalue corresponds to a particular matrix $A$, appearing in the operator $\mathcal{A}$. We shall prove the theorem for an upper endpoint of a spectral gap. The proof for a lower endpoint is identical. We shall require the following lemma.

\begin{lem}
	Consider the operator $\mathcal{A}$ as in~\eqref{eq1:operator}, with $A\in M_B^>$. Let $\lambda_0$ correspond to the upper edge of a spectral gap of $\mathcal{A}$ and let $m$ be the smallest index such that the Bloch eigenvalue $\lambda_{m}$ attains $\lambda_0$, then
	\begin{enumerate}[label={(L\arabic*)}]
		\item There exist numbers $a,b\in\mathbb{R}$ such that $\lambda_{m-1}(\eta)<a<\lambda_0<\lambda_m(\eta)<b$ for all $\eta\in Y^{'}$. Further, there exists $M\in\mathbb{N}$ such that $M>m$ and the Bloch eigenvalue $\lambda_M$ satisfies $\lambda_M(\eta)>b$ for all $\eta\in Y^{'}$.\label{lemma.1}
		\item Let $B$ be a symmetric matrix with $L^\infty_\sharp(Y,\mathbb{R})$-entries. There is a finite open cover of $Y^{'}$, $\{\mathcal{G}_1,\mathcal{G}_2,\ldots,\mathcal{G}_n\}$ such that for each $\mathcal{G}_j$, we have an orthonormal set in $L^2_\sharp(Y)$ of functions analytic for $\eta\in\mathcal{G}_j$ and for sufficiently small $t$,
		\begin{equation}\label{orthonormal set}
		\{ \phi_m^{(j)}(\eta,A-tB), \phi_{m+1}^{(j)}(\eta,A-tB),\ldots, \phi_{R_j}^{(j)}(\eta,A-tB) \}.
		\end{equation}
	    Further, for each fixed $t$, the linear subspace generated by the functions in~\eqref{orthonormal set} contains the eigenspaces corresponding to eigenvalues of $-\nabla\cdot(A-tB)\nabla$ between $a$ and $b$.\label{lemma.2}
		\item The functions in~\eqref{orthonormal set} may be chosen such that the following equation is satisfied
		\begin{equation}\left\langle\frac{d{\phi}^{(j)}_r}{dt},\phi^{(j)}_s\right\rangle=0,\label{orthogonality1}\end{equation}
		where $\langle\cdot,\cdot\rangle$ denotes the $L^2_\sharp(Y)$ inner product.
		\label{lemma.5}
	\end{enumerate}
\end{lem}

\begin{proof}
	\leavevmode
	
	\textbf{Proof of~\ref{lemma.1}} As noted in Remark~\ref{simple_in_neighborhood}, the Bloch eigenvalues are Lipschitz continuous functions on a compact set $\mathbb{T}^d$. Hence, the function $\eta\mapsto\lambda_m(\eta)$ is bounded above, say by $b$. Since, $\lambda_0$ is a spectral edge, $\displaystyle\delta\coloneqq\min_{\eta\in Y^{'}}\lambda_m(\eta)-\max_{\eta\in Y^{'}}\lambda_{m-1}(\eta)$ is positive. Choose $a=\lambda_0-\frac{\delta}{2}$. These choices of $a$ and $b$ satisfy our requirements.

	By Weyl's law~\cite{Reed1978}, the eigenvalues of the periodic Laplacian on $Y$ satisfy the following inequality, for some $s>0$ and $C_1>0$, for large $M$,
	\begin{align}\label{BlochBoundedBelow1}
	\lambda_M(0,I)\geq\lambda_M^{N}\geq C_1 M^s,
	\end{align}
	where $\lambda_M^{N}$ denotes the $M^{th}$ eigenvalue of the Neumann Laplacian on $Y$. 
	
	By Lipschitz continuity of Bloch eigenvalues in the dual parameter, we have 	\begin{align*}|\lambda_M(\eta,I)-\lambda_M(0,I)|\leq C|\eta|\leq C_2.\end{align*} Therefore, for all $\eta\in Y^{'}$, \begin{equation}\label{BlochBoundedBelow2}\lambda_M(\eta,I)\geq \lambda_M(0,I)-C_2.\end{equation} On combining~\eqref{BlochBoundedBelow1} and~\eqref{BlochBoundedBelow2}, for all $\eta\in Y^{'}$, we obtain 	\begin{align*}\lambda_M(\eta,I)\geq C_1M^s-C_2.\end{align*}
	
	It follows from a standard argument involving min-max principle, that $\lambda_M(\eta,I)\leq C_3||A^{-1}||_{L^\infty}\lambda_M(\eta,A)$.
	
	Therefore, for all $\eta\in Y^{'}$, 
	\begin{align*}
	\lambda_M(\eta,A)&\geq \dfrac{1}{C_3||A^{-1}||_{L^\infty}}\lambda_M(\eta,I)\notag\\
	&\quad\geq \dfrac{C_1M^s}{C_3||A^{-1}||_{L^\infty}}-\dfrac{C_2}{C_3||A^{-1}||_{L^\infty}}.
	\end{align*}
	Finally to prove~\ref{lemma.2}, choose $M$ large enough so that 	\begin{align*}\dfrac{C_1M^s}{C_3||A^{-1}||_{L^\infty}}-\dfrac{C_2}{C_3||A^{-1}||_{L^\infty}}>b.\end{align*}
	
	\textbf{Proof of~\ref{lemma.2}}
	For each $\xi\in\mathbb{T}^d$, there is a circle $\Gamma_\xi$ in the complex plane containing the eigenvalues of $\mathcal{A}(\xi)$ between $a$ and $b$. Let $B$ be a $d\times d$ real symmetric matrix with $W^{1,\infty}_\sharp(Y)$ entries. Observe that the operator $P_\xi$ defined by
	\begin{equation}\label{ProjectionOperator}
	P_\xi(\eta;A-tB)\coloneqq-\frac{1}{2\pi i}\int_{\Gamma_\xi} \left(\mathcal{A}(\eta;A-tB)-zI\right)^{-1}~dz
	\end{equation}
	is real-analytic in a neighborhood $R_\xi$ of $\xi$ and for small $t$, where 	\begin{align*}\mathcal{A}(\eta;A-tB)\coloneqq -(\nabla+i\eta)\cdot(A-tB)(\nabla+i\eta).\end{align*} The operator $P_\xi$ is an orthogonal projection onto the eigenspace of $\mathcal{A}(\eta;A-tB)$ corresponding to the eigenvalues between $a$ and $b$. The analyticity of the projection operator follows from the analyticity of the integrand, which is a consequence of the operator family $\mathcal{A}(\eta;A-tB)$ being a holomorphic family of type $(B)$. A proof of this fact is available in~\cite{Sivaji2004} for perturbation in $\eta$. For a perturbation in $t$, a proof is given in Appendix~\ref{PerturbationTheory}.
	
	Therefore, in a neighborhood of $\eta=\xi, t=0$, we obtain an orthonormal basis for the range of $P_{\xi}(\eta,A-tB)$.  In this manner, we obtain an open cover of $\mathbb{T}^d$. By compactness of $\mathbb{T}^d$, the open cover has a finite subcover $\{\mathcal{G}_1,\mathcal{G}_2,\ldots,\mathcal{G}_n\}$ with the following properties.
	\begin{enumerate}
		\item For each $\mathcal{G}_j$, we have an orthonormal set in $L^2_\sharp(Y)$	\begin{align*}\{\phi^{(j)}_{m}(\eta,A-tB),\ldots,\phi^{(j)}_{R_j}(\eta,A-tB)\}\end{align*} whose elements are analytic for $\eta\in\mathcal{G}_j$ and $|t|<\delta$.
		\item The linear subspace generated by 	\begin{align*}\{\phi^{(j)}_{m}(\eta,A-tB),\ldots,\phi^{(j)}_{R_j}(\eta,A-tB)\}\end{align*} contains the eigenspaces corresponding to eigenvalues of $\mathcal{A}(\eta;A-tB)$ that lie between $a$ and $b$.
	\end{enumerate}
	\textbf{Proof of~\ref{lemma.5}}  Let $\tilde\phi_r=\sum_{p=m} u_{rp}\phi_p$, then
	
	\begin{align*}
	\left\langle\frac{d{\tilde{\phi_r}}}{dt},\tilde\phi_s\right\rangle=\sum_p\frac{d{u}_{rp}}{dt}\overline{u_{sp}}+\sum_{p,q}u_{rp}\overline{u_{sq}}\left\langle\frac{d{\phi}_p}{dt},\phi_q\right\rangle.
	\end{align*}
	
	If we set $U$ to be the matrix with entries $u_{rs}$ and $A$ to be the matrix with entries $-\left\langle\phi_r,\frac{d{\phi}_s}{dt}\right\rangle$, \eqref{orthogonality1} will hold if
	\begin{align*}
	\frac{dU}{dt}=UA.
	\end{align*}
	This is solved with the initial condition $U(0)=I$. The matrix $A$ is skew-symmetric, therefore, $U(t)$ is unitary and analytic for $\eta\in\mathcal{G}$. Replace $\phi_r$ with $\tilde{\phi}_r$ to complete the proof of~\ref{lemma.5}.
\end{proof}

\begin{proof}[Proof of Theorem~\ref{theorem:3}]
	\leavevmode
	
	Consider the sesquilinear form 	\begin{align*}a(\eta,t)(u,v)\coloneqq\int_Y(A-tB)(\nabla+i\eta)u\cdot(\nabla-i\eta)\overline{v}.\end{align*} For the functions constructed in~\ref{lemma.2}, $\langle\frac{d\phi^{(j)}_r}{dt},\phi^{(j)}_s\rangle=0$ for all $r,s$. Thus,
	\begin{align*}
		\frac{d}{dt}\left(a(\eta,t)(\phi_r^{(j)}(\eta,t),\phi_s^{(j)}(\eta,t))\right)=-\int_Y B(\nabla+i\eta)\phi_r^{(j)}(\eta,t)\cdot(\nabla-i\eta)\overline{\phi_s^{(j)}(\eta,t)}.
	\end{align*}
	
	A function $f$ defined on $\mathbb{R}^d$ is said to be $(\eta,Y)$-periodic if for all $p\in\mathbb{Z}^d$, $y\in\mathbb{R}^d$, $u(y+2\pi p)=e^{2\pi i p\cdot \eta}u(y)$. The eigenfunctions of $\mathcal{A}(\eta,A)$ with periodic boundary conditions, when multiplied by $\exp(-i\eta\cdot y)$, become eigenfunctions of $\mathcal{A}\coloneqq-\nabla\cdot A\nabla$ with $(\eta,Y)$-periodic boundary conditions, i.e., there are $\lambda$ and $u$ such that $-\nabla\cdot(A\nabla)u=\lambda u$, where $u$ is $(\eta,Y)$-periodic. Since $u$ is a complex-valued function, the regularity theorem~\cite[Chapter 3, Section 15]{Ladyzhenskaya68}, cannot be applied directly. However, since the operator is linear, we may write $u=v+iw$ and express the eigenvalue equation for $u$ as two equations for the real-valued functions $v$ and $w$. In particular, $v$ and $w$ satisfy $-\nabla\cdot(A\nabla)v=\lambda v$ and $-\nabla\cdot(A\nabla)w=\lambda w$ in the interior of $Y$. Hence, by the regularity theory for elliptic equations with $W^{1,\infty}$ coefficients, $v$ and $w$ and their first-order derivatives are H\"older continuous in the interior of $Y$. Further, the H\"older estimates in the interior of $Y$ are independent of $\eta\in Y^{'}$. Consequently, $u$ and its derivatives are H\"older continuous in the interior of $Y$.
	
	Choose $\widehat{\eta}$ and $\phi_0$ such that $\mathcal{A}(\widehat{\eta},A)\phi_0=\lambda_0\phi_0.$ Choose $\phi_0\neq \exp(-i\hat{\eta}\cdot y)$. This can be achieved because the multiplicity of the Bloch eigenvalue at $\hat{\eta}$ is greater than one. Therefore, $(\nabla+i\widehat{\eta})\phi_0$ is non-zero. Consequently, there exist a $y_0$ in the interior of $Y$, an $l$ with $1\leq l\leq d$, and a $\theta>0$ such that $\left|\left(\frac{\partial}{\partial x_l}+i\widehat{\eta}_l\right)\phi_0(y_0,\widehat{\eta})\right|^2\geq\theta$. Since $\phi_0$ and its derivatives are H\"older continuous in the interior of $Y$, there is a small $\epsilon_0>0$ such that, \begin{align}\label{continuityofderivatives1}\mbox{for }|y-y_0|<\epsilon_0,\qquad\left|\left(\frac{\partial}{\partial x_l}+i{\widehat{\eta}_l}\right)\phi_0(y,\widehat{\eta})\right|^2>\frac{2\theta}{3}.\end{align}
	
	Additionaly, since $\phi_r^{(j)}$ obtained earlier in~\ref{lemma.2} are linear combinations of eigenfunctions, by the H\"older continuity of the eigenfunctions and their derivatives, an $\epsilon_0$ may be chosen so that
	\begin{align}\label{continuityofderivatives2}
	\sum_{p=m}^{R_j} \left|\left(\frac{\partial}{\partial x_l}+i{\eta}_l\right)\phi_p^{(j)}(y,\eta,A)-\left(\frac{\partial}{\partial x_l}+i{\eta}_l\right)\phi_p^{(j)}(y_0,\eta,A)\right|^2<\frac{\theta}{3},
	\end{align}
	for $\eta\in\mathcal{G}_j$ and $|y-y_0|<\epsilon_0$. Define the matrix $B=diag(0,\ldots,0,b_l,0,\ldots,0)$ all of whose diagonal entries are zero other than $b_{l}$ which is chosen as a function $b_{l}\in C_0^{\infty}(|y-y_0|<\epsilon_0)$ such that $b_{l}\geq 0$ and $\int_Y b_{l}=1$. Extend $B$ periodically to $\mathbb{R}^d$.
	
	There is an index $q$ such that $\widehat{\eta}\in\mathcal{G}_q$. Therefore, $\displaystyle\phi_0(y,\widehat{\eta})=\sum_{r=m}^{R_q} c_r\phi_r^{(q)}(y,\widehat{\eta},A)$. 
	
	Define $\phi_0(y,\widehat{\eta},t)=\displaystyle\sum_{r=m}^{R_q} c_r\phi_r^{(q)}(y,\widehat{\eta},A-tB)$. Then, by~\eqref{continuityofderivatives1},
	\begin{align*}
		\frac{d}{dt}\left(a(\widehat{\eta},t)(\phi_0(\cdot,\widehat{\eta},t),\phi_0(\cdot,\widehat{\eta},t))\right)|_{t=0}&=-\int_Y b_l\left(\frac{\partial}{\partial y_l}+i\widehat{\eta_l}\right)\phi_0(y,\widehat{\eta})\left(\frac{\partial}{\partial y_l}-i\widehat{\eta_l}\right)\overline{\phi_0(y,\widehat{\eta})}~dy\notag\\
		&\qquad\leq\frac{-2\theta}{3}.
	\end{align*}
	
	Hence, 
	\begin{align}
	\label{eq:var3}
	a(\widehat{\eta},t)(\phi_0(\cdot,\widehat{\eta},t),\phi_0(\cdot,\widehat{\eta},t))\leq\lambda_0-\frac{2\theta}{3}t+t^2\beta(t).
	\end{align}
	
	For each $\eta\in\mathcal{G}_j$, we define the function
	\begin{align}\label{stareq}
	\phi_{*}^{(j)}(y,\eta,t)=\sum_{r=m}^{R_j}\overline{\left({\partial_l}+i{\eta}_l\right)\phi_r^{(j)}(y_0,\eta,A)}\phi_r^{(j)}(y,\eta,t).
	\end{align}
	For $\phi(\cdot,\eta,t)=\displaystyle\sum_{k=r}^{R_j} a_r\phi_r^{(j)}(\cdot,\eta,A-tB)$, $\phi(\cdot,\eta,t)$ is perpendicular to $\phi_*^{(j)}(\cdot,\eta,t)$ if and only if
	\begin{align}
	\label{eq:perp1}
	\sum_{r=1}^{R_j} a_r \left({\partial_l}+i{\eta}_l\right)\phi_r^{(j)}(y_0,\eta,A) = 0.
	\end{align}
	For $\phi(\cdot,\eta,t)$ satisfying~\eqref{eq:perp1} and $||\phi||_{L^2_\sharp(Y)}=1$, the following holds for $\eta\in\mathcal{G}_j$,
	\begin{align}
		{}&\frac{d}{dt}\left(a(\eta,t)(\phi(\cdot,\eta,t),\phi(\cdot,\eta,t))\right)|_{t=0}\\
		&\quad=-\int_Y b_{l}\left(\frac{\partial}{\partial y_l}+i{\eta_l}\right)\phi(y,\eta,0)\left(\frac{\partial}{\partial y_l}-i{\eta_l}\right)\overline{\phi(y,\eta,0)}~dy\notag\\
		&\quad=-\int_{B_{\epsilon_0}(y_0)}\left|\sum_{r=m}^{R_j} a_r\left(\left({\partial_l}+i{\eta}_l\right)\phi_r^{(j)}(y,\eta,A)- \left({\partial_l}+i{\eta}_l\right)\phi_r^{(j)}(y_0,\eta,A)\right)\right|^2b_{l}~dy\notag\\
		&\quad\geq-\frac{\theta}{3},
	\end{align}where the last inequality follows from~\eqref{continuityofderivatives2}.
	Therefore, the following holds true, uniformly for $\eta\in\mathcal{G}_j$ and $||\phi||_{L^2_\sharp(Y)}=1$,
	\begin{align}
	\label{eq:var4}
	a(\eta,t)(\phi(\cdot,\eta,t),\phi(\cdot,\eta,t))\geq\lambda_0-\frac{\theta}{3}t+t^2\gamma(t).
	\end{align}
	
	To find an upper bound for $\lambda(\widehat{\eta},t)$, we apply the following variational characterization of the eigenvalues of $\mathcal{A}(\eta,t)$ to~\eqref{eq:var3}. If $\phi_1,\phi_2,\ldots,\phi_{m-1}$ are the first $m-1$ eigenfunctions corresponding to the selfadjoint operator $\mathcal{A}(\eta,t)$, then the $m^{th}$ eigenvalue of $\mathcal{A}(\eta,t)$ is given by the formula 	\begin{align*}\lambda_{m}(\eta,t)=\min_{\phi\perp\{\phi_1,\phi_2,\ldots,\phi_{m-1}\},~||\phi||_{L^2_\sharp(Y)}=1}~a(\eta,t)(\phi,\phi).\end{align*}
	Therefore, 
	\begin{align}\label{estimate1}
	\lambda_{m}(\widehat{\eta},t)<\lambda_0-\frac{7\theta}{12}t,
	\end{align}
	for $t$ sufficiently small.
	To find a lower bound for $\lambda_{m+1}(\eta,t)$, we apply another variational characterization for the eigenvalues to ~\eqref{eq:var4}, viz.,
	\begin{align}
	\lambda_{m+1}(\eta,t)=\max_{{\dim V = m}}~\min_{{\phi\perp V,~||\phi||_{L^2_\sharp(Y)}=1}}~a(\eta,t)(\phi,\phi),
	\end{align}
	
	where $V$ varies over $m$-dimensional subspaces of $H^1_\sharp(Y)$.
	
	For each fixed $\eta$ and $t$, take the $m$-dimensional subspace $V$ spanned by the first $m-1$ eigenfunctions of $\mathcal{A}(\eta,t)$ and $\phi_*^{(j)}$ as defined in~\eqref{stareq}, i.e.,  	\begin{align*}V=\{\phi_1(\eta,t), \phi_2(\eta,t),\ldots,\phi_{m-1}(\eta,t),\phi_*^{(j)}(\eta,t)\}.\end{align*} Then, $\phi(\eta,t)$ satisfying the equation~\eqref{eq:perp1} is perpendicular to $V$ and allows us to conclude that
	\begin{align}\label{estimate2}
	\lambda_{m+1}(\eta,t)>\lambda_0-\frac{5\theta}{12}t,
	\end{align}
	for small $t$. The two estimates obtained above~\eqref{estimate1} and~\eqref{estimate2} together imply that the perturbed spectral edge is attained by a single Bloch eigenvalue.
\end{proof}

\begin{rem}
	The proof of Theorem~\ref{theorem:3} depends crucially on the interior H\"older continuity of the Bloch eigenfunctions and their derivatives. This requires the coefficients of the elliptic operator to have $W^{1,\infty}_\sharp(Y)$ entries. We attempt to reduce this regularity requirement to $L^\infty$ in Section~\ref{simplicity_edge_2}. 
\end{rem}

\section{Proof of Theorem~\ref{theorem:4}}\label{simplicity_edge_2}

We shall prove Theorem~\ref{theorem:4} for an upper endpoint of a spectral gap. The proof for a lower endpoint is identical. Let $\lambda_0$ be the upper endpoint of a spectral gap of $\mathcal{A}\coloneqq -\nabla\cdot(A\nabla)$, which is achieved by the Bloch eigenvalue $\lambda_m(\eta)$ at finitely many points $\eta_1,\eta_2,\ldots,\eta_N$ in $Y^{'}$. The proof uses ideas from Parnovski and Shterenberg~\cite{ParShteren17} and is divided into the following steps:

\begin{enumerate}
	\item By Proposition~\ref{corollarymultiple}, there is a single perturbation $B$ of the coefficients so that the Bloch eigenvalue $\lambda_m(\eta;A+tB)$ is simple at the points $\eta_1,\eta_2,\ldots,\eta_N$.
	\item However, the perturbation creates new points at which the new spectral edge has been attained. We shall prove that given $\delta>0$, we can find perturbation parameter $t$ such that all the points at which the spectral edge is attained are within $\delta$-distance of the old spectral edge (Lemma~\ref{continuity_spectral_edge}).
	\item We prove that these new spectral edges are not multiple.
\end{enumerate}

We shall require the following preliminaries.

The multiplicity of a Bloch eigenvalue can be reduced at a finite number of points in the dual parameter by application of the same perturbation. This will be the content of the next proposition. 

\begin{prop}\label{corollarymultiple}
	Fix $m\in\mathbb{N}$. Let $S=\{\eta_1,\eta_2,\ldots,\eta_N\}$ be a finite collection of points in $Y^{'}$. Then, there exists a matrix $B$ with $L^\infty_\sharp(Y,\mathbb{R})$-entries and a $t_0$ positive such that for all $t\in(0,t_0]$, the Bloch eigenvalue $\lambda_m(t,\eta)$ of the operator $\mathcal{A}+t\mathcal{B}=-\nabla\cdot(A+tB)\nabla$ is simple for all $\eta_n\in S$, $1\leq n\leq N$.
\end{prop}

To this end, we require the following lemma.

\begin{lem}\label{hahnbanach}
	Let $N\in\mathbb{N}$. Let $X$ be a normed linear space over $\mathbb{K}~(\mathbb{R} \mbox{ or } \mathbb{C})$ and let $x_1,x_2,\ldots,x_N$ be non-zero elements of $X$. Then there exists an $x^*\in X^*$ such that $\forall$ $n=1,2,\ldots,N$, $\langle x^*,x_n\rangle\neq 0$ 
\end{lem}

\begin{proof}
	Consider the finite dimensional subspace $F$ of $X$ spanned by $x_1,x_2,\ldots,x_N$. For each $n=1,2,\ldots,N$, let $F_n^*$ denote the subspace of $F^*$ containing $x^*\in F^*$ such that $\langle x^*,x_n \rangle=0$. Then, $F^*\neq \cup_{n=1}^N F^*_n$ since a vector space cannot be written as a finite union of its proper subspaces. Hence, there exists an $x^*\in F^*$ such that $x^*\not\in\cup_{n=1}^N F^*_n$. Hence, for all $n=1,2,\ldots,N$, $\langle x^*,x_n\rangle\neq 0$. Finally, extend $x^*$ to $X^*$ using the Hahn-Banach Theorem.
\end{proof}

\begin{proof}[Proof of Proposition~\ref{corollarymultiple}]
	As a part of the proof of Lemma~\ref{lemma:22}, we prove that for a given $m\in\mathbb{N}$ and $\eta_0\in Y^{'}$, there exists a $t_0$ positive such that for all $t\in(0,t_0]$, the Bloch eigenvalue $\lambda(t,\eta)$ of the perturbed operator $\mathcal{A}+t\mathcal{B}$ is simple at $\eta_0$. In the present proposition, we shall make a Bloch eigenvalue $\lambda_m(\eta)$ of the operator $\mathcal{A}$ simple at a finite number of points in $Y^{'}$ through a perturbation in the coefficients.
	
	As in the proof of Proposition~\ref{find B}, the perturbation at any $\eta_n \in S$ gives rise to a selfadjoint holomorphic family of type $(B)$, analytic in $\tau\in(-\sigma_0,\sigma_0)$, where $\sigma_0=\frac{\alpha}{2d||B||_{L^\infty}}$. Suppose that the eigenvalue $\lambda_m(\eta_n)$ of the operator $\mathcal{A}(\eta_n)$ has multiplicity $h_n$. For the perturbed operator $-\nabla\cdot(A+\tau B)\nabla$, the eigenvalue $\lambda_m(\eta_n)$ splits into $h_n$ branches. Suppose that the $h_n$ eigenvalues and eigenvectors are given as follows. For $n=1,2,\ldots,N$ and $r=1,2,\ldots,h_n$:
	\begin{align*}\lambda_m^{r}(\tau;\eta_n)=\lambda_m(\eta_n)+\tau a_m^r(\eta_n)+\tau^2\beta_m^r(\tau,\eta_n)\\
	u_m^{r}(\tau;\eta_n)=u_m^r(\eta_n)+\tau v_m^r(\eta_n)+\tau^2 w_m^r(\tau,\eta_n).\end{align*}  
	
	As before, the following system of equations holds true for $n=1,2,\ldots,N$ and $r=1,2,\ldots,h_n$:
	\begin{equation*}
	\int_{Y}B(\nabla+i\eta_n) u_m^r(\eta_n)\cdot\overline{(\nabla+i\eta_n) u_m^s(\eta_n)}~dy=a_m^r(\eta_n)\delta_{rs}.
	\end{equation*}
	
	The above equations define operators that act on the unperturbed eigenspaces at each $\eta_n$. The multiplicity would go down if we find $B$ and bases for the unperturbed eigenspaces in which some off-diagonal entry, in particular, the $(1,2)$-entry is non-zero. To achieve this, we proceed as in the proof of Proposition~\ref{find B}. For any choice of basis of the unperturbed eigenspace at $\eta_n$, we find that either~\eqref{alternative1} or~\eqref{alternative2} holds. However, we cannot use this idea anymore, since, different $\eta_n$ would have different matrices $B$. To remedy this, we notice that, at each $\eta_n=(\eta_{n,1},\eta_{n,2},\ldots,\eta_{n,d})$, for a basis given by $\{f_n^1,f_n^2,\ldots,f_n^{h_n}\}$ either
	\begin{equation}\label{alternative3}
	\sum_{l=1}^d(\partial_l+i\eta_{n,l}){f_n^1}(\partial_l-i\eta_{n,l})\overline{f_n^2}\not\equiv0,
	\end{equation}  
	
	or, if the above sum is zero, then in the modified basis $\{f_n^1,f_n^1+f_n^2,f_n^3,\ldots,f_n^{h_n}\}$,
	\begin{equation}\label{alternative4}
	\sum_{l=1}^d(\partial_l+i\eta_{n,l}){f_n^1}(\partial_l-i\eta_{n,l})\overline{(f_n^1+f_n^2)}=\sum_{l=1}^d|(\partial_l+{i}\eta_{n,l}){f_n^1}|^2\not\equiv0,
	\end{equation}
	provided that $f_n^1\neq \exp(-i\eta_n\cdot y)$.
	
	We can always choose $f_n^1$ to be a function different from $\exp(-i\eta_n\cdot y)$ since at any of the $\eta_n$, we have an eigenspace of dimension greater than $1$. For each $\eta_n$, call the non-zero sum between~\eqref{alternative3} and~\eqref{alternative4} as $p_n$. Further, take either $\Re(p_n)$ or $\Im(p_n)$ depending on whichever is non-zero. If both are non-zero, we may take either one. This will make sure that we have a collection of only real-valued functions.
	
	By the above procedure, we have $N$ elements of $L^1_\sharp(Y,\mathbb{R})$, again labelled as $\{p_1,p_2,\ldots,p_N\}$. By Lemma~\ref{hahnbanach}, there is an $\alpha\in (L^1_\sharp(Y,\mathbb{R}))^*$ such that $\alpha(p_n)\neq 0$ for all $n=1,2,\ldots,N$. By duality, there exists a $\beta\in L^\infty_\sharp(Y,\mathbb{R})$ such that $\alpha(p_n)=\int_Y \beta p_n~dy\neq 0$. 
	
	Define $B=diag(\beta,\beta,\ldots,\beta)$, then either,
	\begin{equation*}
	\Re{\int_{Y}B(\nabla+i\eta_n) f_n^1\cdot\overline{(\nabla+i\eta_n) f_n^2}~dy}\neq 0,
	\end{equation*}
	or 
	\begin{equation*}
	\Re{\int_{Y}B(\nabla+i\eta_n) f_n^1\cdot(\nabla-i\eta_n)(\overline{f_n^1}+\overline{f_n^2})~dy}\neq 0,
	\end{equation*}
	depending on $\eta_n$. 
	
	At the end of this step, the multiplicity of $\lambda_m(\eta)$ at each of the points $\eta_n$ will reduce at least by $1$. We repeat the procedure with the points among $\{\eta_1,\eta_2,\ldots,\eta_N\}$ where the eigenvalue is still multiple. Finally, we require at most $M$ steps to make the Bloch eigenvalue simple at each of these points, where $M=\displaystyle\max_{1\leq n\leq N}{h_n}$.
\end{proof}

In the next lemma, we shall prove that a spectral edge does not move very far for small perturbations in the coefficients of the periodic operator $\mathcal{A}$. We shall denote the operator $-\nabla\cdot(A+tB)\nabla$ as $\mathcal{A}+t\mathcal{B}$, where $\mathcal{A}=-\nabla\cdot(A\nabla)$ and $\mathcal{B}=-\nabla\cdot(B\nabla)$. Let $S_t$ denote the set of points at which the new spectral edge is attained, i.e., 
\begin{align*}
S_t \coloneqq \{\eta\in Y^{'} : \mbox{ The Bloch eigenvalue } \lambda_{m}(\eta;A+tB)\mbox{ attains the spectral edge at } \eta.\}
\end{align*}

\begin{lem}\label{continuity_spectral_edge}
	Let $N\in\mathbb{N}$. Let $A\in M_B^>$ and let $B$ be a real symmetric matrix with $L^\infty_\sharp(Y)$-entries. Let $\mathcal{A}=-\nabla\cdot(A\nabla)$ be a periodic elliptic differential operator. Let $\lambda_0$ be the upper endpoint of a spectral gap, which is attained by the Bloch eigenvalue $\lambda_m(\eta)$ at finitely many points $\eta_1,\eta_2,\ldots,\eta_N$ in $Y^{'}$. Given a $ \delta$ belonging to the open interval $(0,1)$, there is a $t_0$ such that
		\begin{align*}\mbox{for }t\in(0,t_0],\qquad S_t\subset \bigcup_{j=1}^{N}B(\eta_j,\delta).\end{align*}
\end{lem}

\begin{proof}
	We prove this lemma by contradiction. Assume that there is a $\delta\in(0,1)$ and sequences $(t_n)$ and $(\xi_n)$ such that $t_n\to 0$ and $\xi_n\in S_{t_n}$ such that
	\begin{equation}\forall~1\leq j\leq N,\qquad|\xi_n-\eta_j|\geq \delta.\label{farpoint}\end{equation}
	
	Let $\lambda_0(A+tB)$ denote the spectral edge associated to the operator $\mathcal{A}+t\mathcal{B}$. The perturbed spectral edge satisfies the following inequality.
	\begin{align}
		|\lambda_0(A)-\lambda_0(A+t_nB)|& = |\min_{\eta\in Y^{'}}\lambda_m(\eta;A)-\min_{\eta\in Y^{'}}\lambda_m(\eta;A+t_nB)|\notag\\
		&\quad = |-\max_{\eta\in Y^{'}}(-\lambda_m(\eta;A))+\max_{\eta\in Y^{'}}(-\lambda_m(\eta;A+t_nB))|\notag\\
		&\quad\leq\max_{\eta\in Y^{'}}|\lambda_m(\eta;A)-\lambda_m(\eta;A+t_nB)|\notag\\
		&\quad\leq Ct_n.\label{edgeconvergence}
	\end{align}
	
	Since $(\xi_n)$ is a bounded sequence in $Y^{'}$, a subsequence of $(\xi_n)$ converges to $\hat{\xi}$, which we continue to denote by $(\xi_n)$.
	
	We shall prove that
	\begin{align}\label{diagconvergence}
	\lambda_m(\xi_n;A+t_nB)\to\lambda_m(\hat{\xi};A).
	\end{align}

	Observe that,
	\begin{align*}
		{}&\left|\int (A+t_n B)(\nabla+i\xi_n)u(\nabla-i\xi_n)\bar{u}dy-\int A(\nabla+i\hat{\xi})u(\nabla-i\hat{\xi})\bar{u}dy\right|\leq\\
		&\quad\left|\int A(\nabla+i\xi_n)u(\nabla-i\xi_n)\bar{u}dy-\int A(\nabla+i\hat{\xi})u(\nabla-i\hat{\xi})\bar{u}dy\right|+t_n||B||_{L^\infty}\int (\nabla+i\xi_n)u(\nabla-i\xi_n)\bar{u}dy
	\end{align*}
	
	Divide throughout by $||u||^2_{L^2(Y)}$ and apply the min-max principle to obtain the following inequality.
	\begin{align}\label{triangle1}
		|\lambda_m(\xi_n;A+t_nB)-\lambda_m(\hat{\xi};A)|\leq|\lambda_m(\xi_n)-\lambda_m(\hat{\xi})|+t_n||B||_{L^\infty}|\lambda_m(\xi_n;I)|.
	\end{align}
	
	In order to establish~\eqref{diagconvergence}, notice that the first and second part of~\eqref{triangle1}  converge to $0$ by the Lipschitz continuity of $\lambda_m(\cdot)$ and the boundedness of $\lambda_m(\xi_n;I)$, respectively.

	It follows from~\eqref{edgeconvergence} and~\eqref{diagconvergence} that $\lambda_0(A)=\lambda_m(\hat{\xi};A)$  and hence, $\hat{\xi}$ is also a spectral edge. By~\eqref{farpoint}, this contradicts the initial assumption that there are only $N$ points at which the spectral edge is attained.
\end{proof}

\begin{proof}[Proof of Theorem~\ref{theorem:4}]
	The spectral edge of the operator $\mathcal{A}$ is attained at finitely many points $\eta_1,\eta_2,\ldots,\eta_N$ in $Y^{'}$. Choose among $\eta_1,\eta_2,\ldots,\eta_N$ the points where the Bloch eigenvalue $\lambda_m(\eta)$ is not a simple eigenvalue. Now, apply Proposition~\ref{corollarymultiple} to these points, so that for the perturbed operator $\mathcal{A}+t\mathcal{B}$, the corresponding Bloch eigenvalue becomes simple at these points. The points which were simple to begin with, will remain simple for sufficiently small $t$.
	
	There is a neighborhood $\mathcal{O}_j$ of each of the points $(\eta_j)_{j=1}^N$ in which the Bloch eigenvalue is simple for a range of $t$. Each of these neighborhoods contain a ball, $B(\eta_j,\delta_j)$ of radius $\delta_j$ centered at $\eta_j$. Let $\delta\coloneqq\displaystyle\min_{1\leq j\leq N}\delta_j$, then by Lemma~\ref{continuity_spectral_edge}, there exists a $t_0$ positive such that for all $t\in(0,t_0]$, the spectral edge of the perturbed operator $\mathcal{A}+t\mathcal{B}$ is contained in the union of the balls $\displaystyle\bigcup_{j=1}^N B(\eta_j,\delta)$. 
	
	Hence, we have obtained a perturbation of the operator $\mathcal{A}$ such that its spectral edge is simple.
\end{proof}

\section{An Application to the Theory of Homogenization}\label{homogen}

Birman and Suslina~\cite{BirmanSuslina2003} have described homogenization as a {\it spectral threshold effect}. Their analysis focuses on finding norm resolvent estimates of different orders. For the operator $\mathcal{A}$, it is known that $\inf \sigma(\mathcal{A})=0$. This corresponds to the bottom edge of its spectrum. A non-zero spectral edge is called an internal edge. The notion of homogenization has been extended to internal edges in~\cite{Birman2004},~\cite{Birman2006}.

\subsection{Internal Edge Homogenization}\label{internal1}

In this subsection, we review the internal edge homogenization theorem of Birman and Suslina~\cite{Birman2006}. Consider the  equation~\eqref{anotherequation} corresponding to the operator $\mathcal{A}$~\eqref{eq1:operator}. Let $\lambda_0$ denote an internal edge, corresponding to the upper endpoint of a spectral gap of $\mathcal{A}$ and let $m$ be the smallest index such that the Bloch eigenvalue $\lambda_{m}$ attains $\lambda_0$, then
\begin{align*}
	\lambda_0=\min_{\eta\in Y^{'}}\lambda_{m}(\eta).
\end{align*}

Birman and Suslina~\cite{Birman2006} make the following regularity assumptions on $\lambda_0$. These are exactly the properties of spectral edge that are required in order to define effective mass in the theory of motion of electrons in solids~\cite{Filonov15}.

\begin{enumerate}[label={(B\arabic*)}]
	\item\label{assumptions1:1} $\lambda_0$ is attained by the $m^{th}$ Bloch eigenvalue $\lambda_{m}(\eta)$ at finitely many points $\eta_1,\eta_2,\ldots,\eta_N$.
	\item\label{assumptions1:2}For $j=1,2,\ldots,N$, $\lambda_{m}(\eta)$ is simple in a neighborhood of $\eta_j$, therefore, $\lambda_{m}(\eta)$ is analytic in $\eta$ near $\eta_j$.
	\item\label{assumptions1:3} For $j=1,2,\ldots,N$, $\lambda_{m}(\eta)$ is non-degenerate at $\eta_j$, i.e., 
	\begin{align*}
		\lambda_{m}(\eta)-\lambda_0=(\eta-\eta_j)^T B_j (\eta-\eta_j)+{O}(|\eta-\eta_j|^3),\mbox{ for } \eta \mbox{ near } \eta_j,
	\end{align*}
 where $B_j$ are positive definite matrices.
\end{enumerate}

Under these assumptions, the internal edge homogenization theorem is proved.

\begin{thm}[\cite{Birman2006}]\label{BirmanSuslina} Let $\mathcal{A}$ be the operator in $L^2(\mathbb{R}^d)$ defined by~\eqref{eq1:operator} and let $\lambda_0$ be an internal edge of the spectrum of $\mathcal{A}$. Assume conditions~\ref{assumptions1:1},~\ref{assumptions1:2},~\ref{assumptions1:3} and let $\varkappa^2>0$ be small enough so that $\lambda_0-\varkappa^2$ is in the spectral gap. Let $\mathcal{A}^\epsilon$ denote the unbounded operator $-\nabla\cdot\left(A(\frac{x}{\epsilon})\nabla\right)$ defined in $L^2(\mathbb{R}^d)$. For $1\leq j\leq N$, let $\psi_j(y,\eta_j)\coloneqq\exp(iy\cdot\eta_j)\phi_j(y)$, where $\phi_j$ is the eigenvector corresponding to the eigenvalue $\lambda_{0}=\lambda_{m}(\eta_j)$ of the operator $\mathcal{A}(\eta_j)=-(\nabla+i\eta_j)\cdot A(\nabla+i\eta_j)$. Then,
	\begin{align*}
	||R(\epsilon)-R^0(\epsilon)||_{L^2(\mathbb{R}^d)\to L^2(\mathbb{R}^d)}={O}(\epsilon)\quad\mbox{as}\quad\epsilon\to 0\quad&\mbox{where},\\
	R(\epsilon)=\left(\mathcal{A}^\epsilon-(\epsilon^{-2}\lambda_{0}-\varkappa^2)I\right)^{-1} \mbox{and}\quad R^0(\epsilon)\coloneqq&|Y|\sum_{j=1}^N [\psi_j^\epsilon]\left(B_j\nabla^2+\varkappa^2I\right)^{-1}[\overline{\psi_j^\epsilon}]\end{align*} are bounded operators on $L^2(\mathbb{R}^d)$ and $||\cdot||_{L^2(\mathbb{R}^d)\to L^2(\mathbb{R}^d)}$ denotes the operator norm. Here, $[f]$ denotes the operation of multiplication by the function $f$.
\end{thm}

\subsection{Internal Edge Homogenization for a multiple spectral edge} \label{internal2}

In this section, we shall prove a theorem corresponding to internal edge homogenization of the operator $\mathcal{A}^\epsilon=-\nabla\cdot\left(A(\frac{x}{\epsilon})\nabla\right)$ in $L^2(\mathbb{R}^d)$ in the presence of multiplicity. We shall interpret the three assumptions~\ref{assumptions1:1},~\ref{assumptions1:2},~\ref{assumptions1:3} that have been made on the spectral edge as hypotheses on the shape and structure of the spectral edge. Without knowledge of the shape and structure of the spectral edge, it is not possible to obtain any explicit homogenization result. 

Starting with a spectral edge which is not simple, we shall appeal to Theorem~\ref{theorem:4} to modify the spectral edge so that it becomes simple. We shall make the following assumptions on the spectral edge. We assume the finiteness of the number of points at which the spectral edge is attained, however, since the contributions from different points are added up, we may as well assume that the spectral edge is attained at one point. Therefore, suppose that for the operator~\eqref{eq1:operator}, a spectral gap exists. Let $\lambda_0$ denote the upper endpoint of this spectral gap of $\mathcal{A}$ and let $m$ be the smallest index such that the Bloch eigenvalue $\lambda_{m}$ attains $\lambda_0$, then $\displaystyle\lambda_0=\min_{\eta\in Y^{'}}\lambda_{m}(\eta)$.

Suppose that the spectral edge is attained at a unique point $\eta_0\in Y^{'}$. Also suppose that the eigenvalue $\lambda_0$ has multiplicity $2$. Therefore, there exists a neighborhood $\mathcal{O}$ of $\eta_0$, on which the Bloch eigenvalue $\lambda_{m}(\eta)$ is simple except at $\eta_0$. Now, a perturbation matrix $B$ with $L^\infty_\sharp(Y,\mathbb{R})$ entries, as in Theorem~\ref{theorem:4}, is applied to the coefficients of operator $\mathcal{A}$, so that the new operator $\tilde{\mathcal{A}}(t)=\mathcal{A}+t\mathcal{B}$, has a simple spectral edge $\tilde{\lambda}_0(t)$ for sufficiently small $t$. However, the perturbed Bloch eigenvalues $\tilde{\lambda}_{m}(\eta,t)$ and $\tilde{\lambda}_{m+1}(\eta,t)$ are simple in the neighborhood $\mathcal{O}$ for small enough $t$. These properties follow from the analyticity of the projection operator~\eqref{ProjectionOperator}, $P(\eta;A+tB)$, which is a consequence of the operator family $\tilde{\mathcal{A}}(t)$ being a holomorphic family of type $(B)$. For more details, see Appendix~\ref{PerturbationTheory}.

For the perturbed spectral edge, we assume the following hypothesis

\begin{enumerate}[label={(C\arabic*)}]
	\item\label{assumptions2:1} $\tilde{\lambda}_{m}(\eta;t)$ attains minimum $\tilde{\lambda}_0(t)$ at a unique point $\eta_0(t)\in\mathcal{O}$ and is non-degenerate on $\mathcal{O}$, i.e., 
	\begin{align*}
	\tilde{\lambda}_{m}(\eta;t)-\tilde{\lambda}_0(t)=(\eta-\eta_0(t))^T \tilde{B}_0(t) (\eta-\eta_0(t))+{O}(|\eta-\eta_0(t)|^3),
	\end{align*}
	for $\eta\in\mathcal{O}$, where $\tilde{B}_0(t)$ is positive definite, i.e., there is $\alpha_0>0$, independent of $t$, such that $\tilde{B}_0(t)>\alpha_0 I$. Further, the order above holds uniformly for sufficiently $t$.
	\item\label{assumptions2:2} $\tilde{\lambda}_{m+1}(\eta;t)$ attains  minimum $\tilde{\lambda}_1(t)$ at a unique point $\eta_1(t)\in\mathcal{O}$ and is non-degenerate on $\mathcal{O}$, i.e., 
	\begin{align*}
	\tilde{\lambda}_{m+1}(\eta;t)-\tilde{\lambda}_1(t)=(\eta-\eta_1(t))^T \tilde{B}_1(t) (\eta-\eta_1(t))+{O}(|\eta-\eta_1(t)|^3),
	\end{align*}
	for $\eta\in\mathcal{O}$, where $\tilde{B}_1(t)$ is positive definite, i.e., there is $\alpha_1>0$, independent of $t$, such that $\tilde{B}_1(t)>\alpha_1 I$. Further, the order above holds uniformly for sufficiently $t$.
\end{enumerate}

In essence, we are asking for the Bloch eigenvalues to have the shapes before and after the perturbation as in Fig.~\ref{figure2} and Fig.~\ref{figure3}.

\begin{figure}
	\centering
	\includegraphics[scale=0.5]{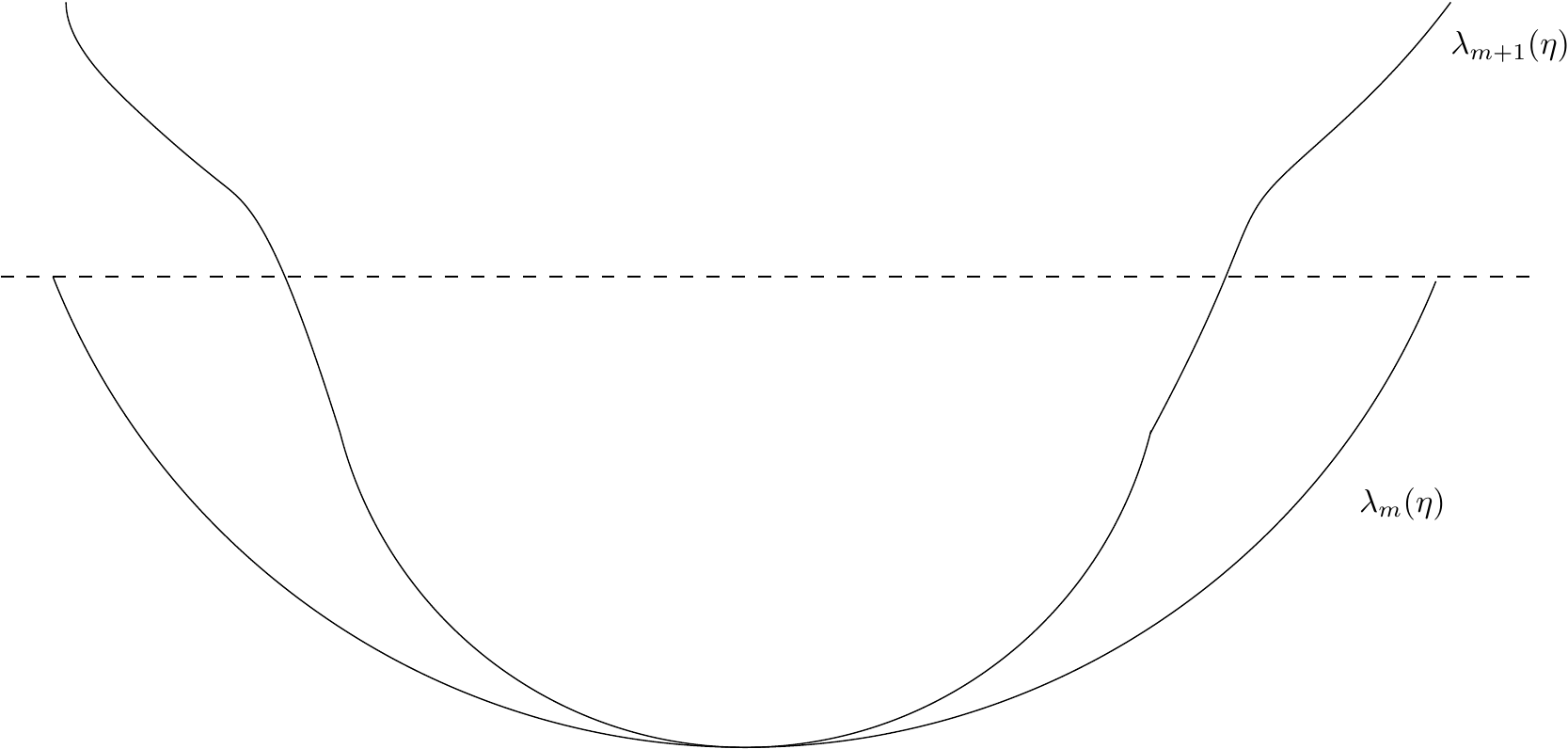}	
	\caption{Spectral Edge before perturbation.}\label{figure2}
\end{figure}

\begin{figure}
	\centering
	\includegraphics[scale=0.5]{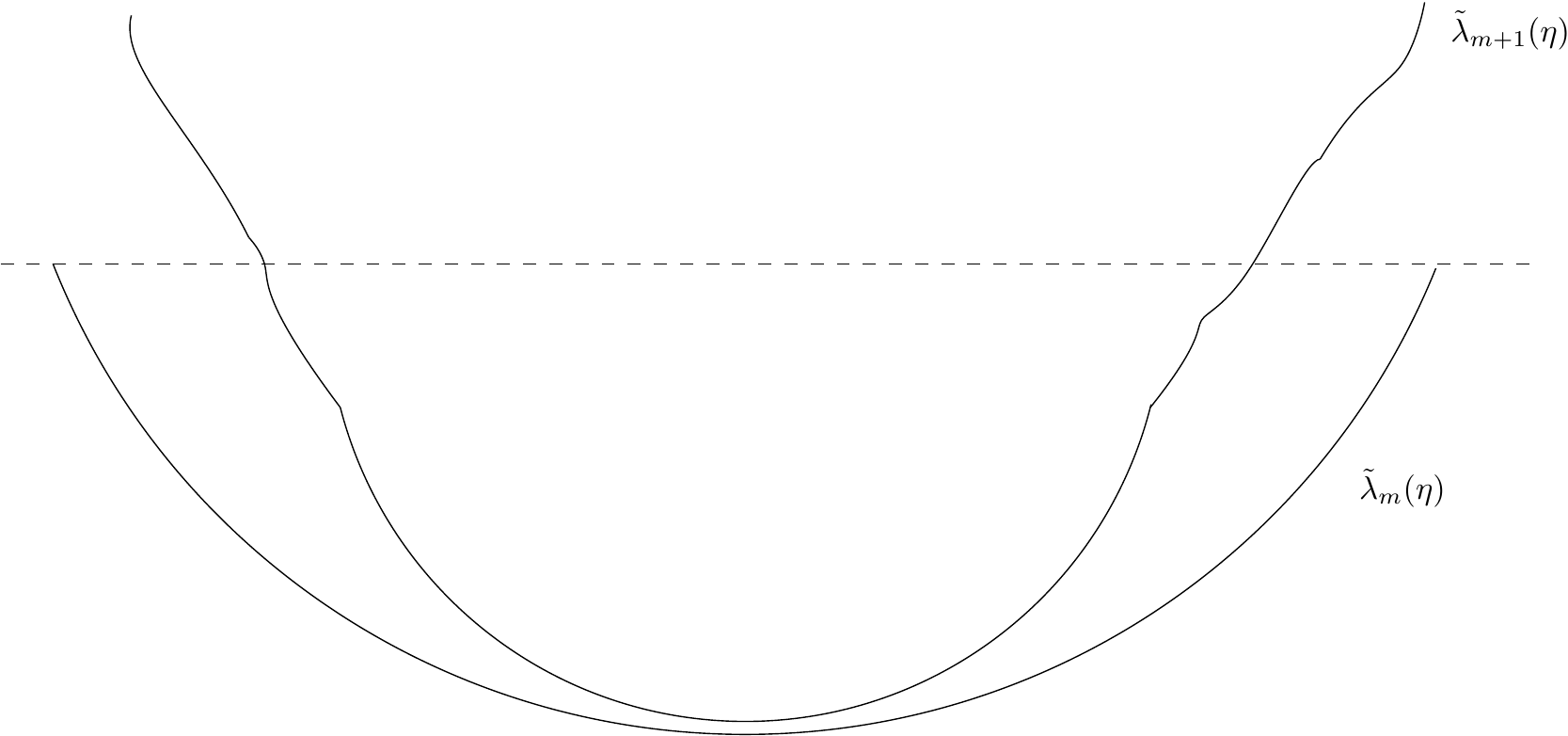}	
	\caption{Spectral Edge after perturbation.}\label{figure3}
\end{figure}
We will now set up notation for the internal edge homogenization theorem that we intend to prove. For $j=0,1$, let $\tilde{\psi}_{m+j}(y,\eta_j(t))=\exp(iy\cdot\eta_j(t))\tilde{\phi}_{m+j}(y;t)$, where $\tilde{\phi}_{m+j}$ is a normalized eigenvector corresponding to the eigenvalue $\tilde{\lambda}_{j}(t)=\tilde{\lambda}_{m+j}(\eta_j(t))$ of $\tilde{\mathcal{A}}(\eta_j;t)=-(\nabla+i\eta_j)\cdot (A+tB)(\nabla+i\eta_j)$. In what follows, we shall choose $t={O}(\epsilon^4)$. Define the following operators
\begin{align}
\label{resolvent1}
{R}(\epsilon)\coloneqq\left(\mathcal{A}^\epsilon-(\epsilon^{-2}\lambda_{0}-\varkappa^2)I\right)^{-1},\mbox{ and}
\end{align}
\begin{align}
	\label{resolvent2}
	\tilde{R}^0(\epsilon)\coloneqq|Y|[\tilde{\psi}_m^\epsilon]\left(\tilde{B}_0(t)\nabla^2+\varkappa^2I\right)^{-1}[\overline{\tilde{\psi}_m^\epsilon}]+|Y|[\tilde{\psi}_{m+1}^\epsilon]\left(\tilde{B}_1(t)\nabla^2+\varkappa^2I\right)^{-1}[\overline{\tilde{\psi}_{m+1}^\epsilon}].
\end{align}

We shall require the following two lemmas.

\begin{lem}\label{resolventlemma1}
	Let 
	\begin{align}
	\label{resolvent3}
	\tilde{{R}}(\epsilon)\coloneqq\left(\tilde{\mathcal{A}}^\epsilon(t)-(\epsilon^{-2}\tilde{\lambda}_{0}(t)-\varkappa^2)I\right)^{-1},
	\end{align} 
	where $\tilde{\mathcal{A}}^\epsilon(t)=-\nabla\cdot\left(A(\frac{x}{\epsilon})+tB(\frac{x}{\epsilon})\right)\nabla$ is an unbounded operator in $L^2(\mathbb{R}^d)$, satisfying assumptions~\ref{assumptions2:1} and~\ref{assumptions2:2}. Choose $t={O}(\epsilon^4)$.  Then,
	\begin{equation*}
	||R(\epsilon)-\tilde{R}(\epsilon)||_{L^2(\mathbb{R}^d)\to L^2(\mathbb{R}^d)}={O}(\epsilon)\quad\mbox{as}\quad \epsilon\to 0.
	\end{equation*}
\end{lem}

\begin{lem}\label{resolventlemma2}
	With the same notation as in Lemma~\ref{resolventlemma1}, it holds that
	\begin{equation*}
	||\tilde{R}(\epsilon)-\tilde{R}^0(\epsilon)||_{L^2(\mathbb{R}^d)\to L^2(\mathbb{R}^d)}={O}(\epsilon)\quad\mbox{as}\quad \epsilon\to 0.
	\end{equation*}
\end{lem}

The proofs of these lemmas will be the content of subsections~\ref{continuity_of_resolvents} and~\ref{internal_homogenization_result}. Now, we state the internal edge homogenization theorem for a multiple spectral edge.

\begin{thm}\label{theorem:5}
	Let $\mathcal{A}$ be the operator defined in $L^2(\mathbb{R}^d)$ as $\mathcal{A}\coloneqq -\nabla\cdot(A\nabla)$. Suppose that the entries of the matrix $A$ belong to $M_B^>$. Let $\lambda_0$ be the upper edge of a spectral gap associated to operator $\mathcal{A}$. Suppose that $\lambda_0$ is attained at one point $\eta_0\in Y^{'}$ and its multiplicity is $2$. Let $\varkappa^2>0$ be small enough so that $\lambda_0-\varkappa^2$ remains in the spectral gap. Let $\mathcal{A}^\epsilon$ be defined as $\mathcal{A}^\epsilon=-\nabla\cdot\left(A(\frac{x}{\epsilon})\nabla\right)$ in $L^2(\mathbb{R}^d)$. 
	
	Let $\tilde{\mathcal{A}}(t)=\mathcal{A}+t\mathcal{B}$ be a perturbation of $\mathcal{A}$ such that the perturbed operator has a simple spectral edge at $\tilde{\lambda}_0(t)$. Let $\tilde{\mathcal{A}}^\epsilon(t)=-\nabla\cdot\left(A(\frac{x}{\epsilon})+tB(\frac{x}{\epsilon})\right)\nabla$. Choose $t={O}(\epsilon^4)$. Assume conditions~\ref{assumptions2:1},~\ref{assumptions2:2} on the perturbed eigenvalues. Then,
	\begin{align}\label{resolventinequality0}
	||R(\epsilon)-\tilde{R}^0(\epsilon)||_{L^2(\mathbb{R}^d)\to L^2(\mathbb{R}^d)}={O}(\epsilon)\quad\mbox{as}\quad \epsilon\to 0,
	\end{align}
	where $R(\epsilon)$ and $\tilde{R}^0(\epsilon)$ are defined in~\eqref{resolvent1} and~\eqref{resolvent2}, respectively.
\end{thm}

\begin{proof}[Proof of Theorem~\ref{theorem:5}]
	Observe that
	\begin{align}\label{intermediateinequality}
		{}&||R(\epsilon)-\tilde{R}^0(\epsilon)||_{L^2(\mathbb{R}^d)\to L^2(\mathbb{R}^d)} \notag\\ &\quad\quad\leq||R(\epsilon)-\tilde{R}(\epsilon)||_{L^2(\mathbb{R}^d)\to L^2(\mathbb{R}^d)}+||\tilde{R}(\epsilon)-\tilde{R}^0(\epsilon)||_{L^2(\mathbb{R}^d)\to L^2(\mathbb{R}^d)}.
	\end{align}
	Applying Lemmas~\ref{resolventlemma1} and~\ref{resolventlemma2} to~\eqref{intermediateinequality}, we obtain~\eqref{resolventinequality0}.
\end{proof}

\begin{rem}
	\leavevmode
	\begin{enumerate}
		\item Theorem~\ref{theorem:5} allows the computation of the homogenized coefficients through perturbed Bloch eigenvalues. Both the crossing modes contribute to homogenization, even though the spectral edge is simple after the perturbation.
		\item A perturbation of the form $\tilde{\mathcal{A}}(t)$, as mentioned in Theorem~\ref{theorem:5}, exists for sufficiently small $t$ by Theorem~\ref{theorem:4}.
		\item If the spectral edge is attained at finitely many points, the contribution to the effective operator from each of those points, are merely added up, as in Theorem~\ref{BirmanSuslina}. Hence, our assumption that the spectral edge is attained at one point is not restrictive. Further, the assumption that multiplicity of the spectral edge is $2$ can also be relaxed, since our method allows successive reduction of multiplicity of Bloch eigenvalues at multiple points.
	\end{enumerate}
\end{rem}

\subsection{Proof of Lemma~\ref{resolventlemma1}}\label{continuity_of_resolvents} The aim of this section is to prove Lemma~\ref{resolventlemma1}. We begin by introducing some notation. 
Define the two resolvents $S(\epsilon)$ and $\tilde{S}(\epsilon)$ by
\begin{align}
\label{scaledresolvents}
S(\epsilon)=\left(\mathcal{A}-(\lambda_{0}-\epsilon^2\varkappa^2)I\right)^{-1}
\quad\mbox{and}\quad
\tilde{S}(\epsilon)=\left(\tilde{\mathcal{A}}(t)-(\tilde{\lambda}_{0}(t)-\epsilon^2\varkappa^2)I\right)^{-1}
\end{align}
Define
\begin{align*}
	\mathfrak{h}[u]\coloneqq\int_{\mathbb{R}^d}A\nabla u\cdot\nabla \overline{u}~dy-\lambda_0\int_{\mathbb{R}^d}|u|^2~dy.
\end{align*}
Then, $\mathfrak{h}$ is a closed sectorial form with domain $H^1(\mathbb{R}^d)$.

Consider another form $\mathfrak{p}(t)$ with domain $H^1(\mathbb{R}^d)$ defined by \begin{align*}
	\mathfrak{p}(t)[u]\coloneqq\int_{\mathbb{R}^d}tB\nabla u\cdot\nabla \overline{u}~dy-(\tilde{\lambda}_0-\lambda_0)\int_{\mathbb{R}^d}|u|^2~dy.
\end{align*}

To the sectorial forms $\mathfrak{h}$ and $\mathfrak{p}$, we shall apply the following theorem about continuity of resolvents which can be found in~\cite[p.~340]{Kato1995}.

\begin{thm}[~\cite{Kato1995}]\label{KLMN}
	Let $\mathfrak{h}$ be a densely defined, closed sectorial form bounded from below and let $\mathfrak{p}$ be a form relatively bounded with respect to $\mathfrak{h}$, so that $D(\mathfrak{h})\subset D(\mathfrak{p})$ and
	\begin{align}\label{eq:hypothesis1}
	|\mathfrak{p}[u]|\leq a||u||^2+b\mathfrak{h}[u],
	\end{align}
	
	where $0\leq b<1$, but $a$ may be positive, negative or zero. Then $\mathfrak{h}+\mathfrak{p}$ is sectorial and closed. Let $H, K$ be the operators associated with $\mathfrak{h}$ and $\mathfrak{h}+\mathfrak{p}$, respectively. If $\zeta$ is not in the spectrum of $H$ and 
	\begin{align}\label{eq:hypothesis2}
	||(a+bH)R(\zeta,H)||<1,
	\end{align}
	then $\zeta$ is not in the spectrum of $K$ and
	\begin{align}\label{inequality5}
	||R(\zeta,K)-R(\zeta,H)||\leq \frac{4||(a+bH)R(\zeta,H)||}{(1-||(a+bH)R(\zeta,H)||)^2}||R(\zeta,H)||.
	\end{align}
\end{thm}\qed

In order to apply the theorem, we must verify the hypotheses~\eqref{eq:hypothesis1} and~\eqref{eq:hypothesis2}.
We shall prove that $\mathfrak{p}$ is relatively bounded with respect to $\mathfrak{h}$, i.e., there exist $a,b\in\mathbb{R}$, such that:
\begin{align*}
	|\mathfrak{p}[u]|\leq a||u||^2+b\mathfrak{h}[u],
\end{align*}

Observe that

\begin{align*}
	\mathfrak{h}[u]\geq \alpha\int_{\mathbb{R}^d}|\nabla u|^2~dy-\lambda_0\int_{\mathbb{R}^d}|u|^2~dy,
\end{align*}
and
\begin{align*}
	\mathfrak{p}[u]&\leq t||B||_{L^\infty}\int_{\mathbb{R}^d}|\nabla u|^2~dy+|\tilde{\lambda}_0-\lambda_0|\int_{\mathbb{R}^d}|u|^2~dy\notag\\
	&\quad=\frac{t||B||_{L^\infty}}{\alpha}\left\{\int_{\mathbb{R}^d}\alpha|\nabla u|^2~dy-\lambda\int_{\mathbb{R}^d}|u|^2~dy\right\}+\left\{|\tilde{\lambda}_0-\lambda_0|+\frac{t||B||_{L^\infty}}{\alpha}\lambda_0\right\}\int_{\mathbb{R}^d}|u|^2~dy\notag\\
	&\quad= b\mathfrak{h}[u]+a||u||^2,
\end{align*}

where $a=\left\{|\tilde{\lambda}_0-\lambda_0|+\frac{t||B||_{L^\infty}}{\alpha}\lambda_0\right\}\approx c_1t$ and $b=\frac{t||B||_{L^\infty}}{\alpha}=c_2t$ for some constants $c_1$ and $c_2$.

Next, observe that for selfadjoint operator $H$, the resolvent $R(\zeta,H)$ is a normal operator, therefore, we have (see~\cite[p.~177]{Kato1995}) \begin{align*}||R(\zeta,H)||\leq\dfrac{1}{dist(\zeta,\sigma(H))}.\end{align*} Further,
\begin{align*}
	||(a+bH)R(\zeta,H)||&\leq||aR(\zeta,H)||+||bHR(\zeta,H)||&\notag\\
	&\quad\leq\frac{a}{dist(\zeta,\sigma(H))}+||b(I-\zeta R(\zeta,H))||&\notag\\
	&\quad\leq\frac{a}{dist(\zeta,\sigma(H))}+b||I||+b||\zeta R(\zeta,H)||&\notag\\
	&\quad\leq\frac{a}{dist(\zeta,\sigma(H))}+b+b\frac{|\zeta|}{dist(\zeta,\sigma(H))}.
\end{align*} 

The operator corresponding to the sectorial form $\mathfrak{h}$ is $H\coloneqq-\nabla\cdot A\nabla-\lambda_0I$, therefore, $0\in\sigma(H)$, so that, for $\zeta=-\epsilon^2\varkappa^2$
\begin{align*}
	||(a+bH)R(\zeta,H)||&\leq\frac{a}{\epsilon^2\varkappa^2}+2b.&
\end{align*} 

Notice that $R(\zeta,H)=S(\epsilon)$ and $R(\zeta,K)=\tilde{S}(\epsilon)$. Let us assume that $t$ is small enough so that Theorem~\ref{KLMN} can be applied to the resolvents in~\eqref{scaledresolvents}. In particular, we have
\begin{align*}
	||S(\epsilon)-\tilde{S}(\epsilon)||_{L^2(\mathbb{R}^d)\to L^2(\mathbb{R}^d)}&=~||R(\zeta,H)-R(\zeta,K)||&\notag\\&\leq\dfrac{4(c_1t+2c_2t\epsilon^2\varkappa^2)}{(\epsilon^2\varkappa^2-c_1t-2c_2t\epsilon^2\varkappa^2)^2}.&
\end{align*}

Choose $t$ so that $c_1t=\epsilon^4\varkappa^2$, then,
\begin{align*}
	||S(\epsilon)-\tilde{S}(\epsilon)||_{L^2(\mathbb{R}^d)\to L^2(\mathbb{R}^d)}\leq\frac{4(1+2c_3\epsilon^2\varkappa^2)}{\varkappa^2(1-\epsilon^2-2c_3\epsilon^4\varkappa^2)^2}.
\end{align*}

Further, for $\epsilon^2<1/2$,
\begin{align}\label{some_inequality_1}
||S(\epsilon)-\tilde{S}(\epsilon)||_{L^2(\mathbb{R}^d)\to L^2(\mathbb{R}^d)}\leq\frac{16(1+c_3\varkappa^2)}{\varkappa^2(1-c_3\varkappa^2)^2}.
\end{align}

\begin{proof}[Proof of Lemma~\ref{resolventlemma1}]
	Define the scaling transformation $T_\epsilon$ by
	\begin{align*}
		T_\epsilon:u(y)\mapsto \epsilon^{d/2}u(\epsilon y).
	\end{align*}
	These are unitary operators on $L^2(\mathbb{R}^d)$.
	For the operators~\eqref{resolvent1} and~\eqref{resolvent3}, it holds that
	\begin{align*}
		R(\epsilon)=\epsilon^2T^*_\epsilon S(\epsilon)T_\epsilon\quad\mbox{and}\quad
		\tilde{R}(\epsilon)=\epsilon^2T^*_\epsilon \tilde{S}(\epsilon)T_\epsilon.
	\end{align*}
	
	Proving Lemma~\ref{resolventlemma1} is equivalent to proving that
	\begin{align*}
		||S(\epsilon)-\tilde{S}(\epsilon)||_{L^2(\mathbb{R}^d)\to L^2(\mathbb{R}^d)}={O}\left(\dfrac{1}{\epsilon}\right).
	\end{align*}
	In fact, in~\eqref{some_inequality_1}, we proved \begin{align*}
		||S(\epsilon)-\tilde{S}(\epsilon)||_{L^2(\mathbb{R}^d)\to L^2(\mathbb{R}^d)}={O}(1).
	\end{align*}\end{proof}

\subsection{Proof of Lemma~\ref{resolventlemma2}}\label{internal_homogenization_result} The aim of this section is to prove Lemma~\ref{resolventlemma2}. Let $\left(\tilde{\lambda}_l(\eta;t)\right)_{l=1}^\infty$ and $\left(\tilde{\phi}_l(y,\eta;t)\right)_{l=1}^\infty$ be the Bloch eigenvalues and the corresponding orthonormal Bloch eigenvectors for the operator $\tilde{\mathcal{A}}(t)$, defined in Theorem~\ref{theorem:5}. Let, $\tilde{\psi}_l(y,\eta;t)=e^{iy\cdot\eta}\tilde{\phi}_l(y,\eta;t)$. In the sequel, we shall suppress the dependence on $t$ for notational convenience. The operator $\tilde{\mathcal{A}}$ may be decomposed in terms of the Bloch eigenvalues as in the theorem below, a proof of which may be found in~\cite{Bensoussan2011}.

\begin{thm}
	Let $g\in L^2(\mathbb{R}^d)$. Define $l^{\,th}$ Bloch coefficient of $g$ as follows:
	\begin{align*}
		(\tilde{\mathcal{B}}_l g)(\eta)=\int_{\mathbb{R}^d}\overline{\tilde{\psi}_l(y,\eta)}g(y)~dy,~l\in\mathbb{N},\eta\in Y^{'}.
	\end{align*}
	
	Then, the following inverse formula holds.
	\begin{align*}
		g(y)=\sum_{l=1}^{\infty}\int_{Y^{'}}(\tilde{\mathcal{B}}_l g)(\eta){\psi}_l(y,\eta)~d\eta=&\sum_{l=1}^{\infty}(\tilde{\mathcal{B}}_l^*)(\tilde{\mathcal{B}}_l g),\quad\mbox{where},\\ (\tilde{\mathcal{B}}_l^*h)(y)=\int_{Y^{'}}h(\eta)\tilde{\psi}_l(y,\eta)\,d\eta&\quad\mbox{for}\quad h\in L^2(Y^{'}).\end{align*} 
	
	In particular, the following representation holds for the operator $\tilde{\mathcal{A}}$:
	\begin{align*}
		\tilde{\mathcal{A}}=\sum_{l\in\mathbb{N}}\tilde{\mathcal{B}}_l^*\tilde{\lambda}_l\tilde{\mathcal{B}}_l.
	\end{align*}
	
	Also,
	\begin{align*} R(\zeta,\tilde{\mathcal{A}})=\left(\tilde{\mathcal{A}}-\zeta I\right)^{-1}=\sum_{l\in\mathbb{N}}\tilde{\mathcal{B}}_l^*(\tilde{\lambda}_l-\zeta)^{-1}\tilde{\mathcal{B}}_l.
	\end{align*}
\end{thm}\qed

Define the Fourier Transform and the inverse Fourier Transform
\begin{align*}
	\left({\mathcal{F}} u\right)(\eta)=\frac{1}{(2\pi)^{d/2}}\int_{\mathbb{R}^d}e^{-iy\cdot\eta}u(y)~dy,\qquad\left({\mathcal{F}^{-1}} u\right)(\eta)=\frac{1}{(2\pi)^{d/2}}\int_{\mathbb{R}^d}e^{iy\cdot\eta}u(y)~dy.
\end{align*}

\begin{proof}[Proof of Lemma~\ref{resolventlemma2}]
	
	Define the operator
	\begin{align*}
		\tilde{S}^0(\epsilon)\coloneqq|Y|[\tilde{\psi}_m]\left(\tilde{{B}}_0\nabla^2+\epsilon^2\varkappa^2I\right)^{-1}[\overline{\tilde{\psi}_m}]+|Y|[\tilde{\psi}_{m+1}]\left(\tilde{{B}}_1\nabla^2+\epsilon^2\varkappa^2I\right)^{-1}[\overline{\tilde{\psi}_{m+1}}]
	\end{align*}
	For the operators~\eqref{resolvent2} and~\eqref{resolvent3}, it holds that
	\begin{align*}
		\tilde{R}(\epsilon)=\epsilon^2T^*_\epsilon \tilde{S}(\epsilon)T_\epsilon\quad\mbox{and}\quad
		\tilde{R}^0(\epsilon)=\epsilon^2T^*_\epsilon \tilde{S}^0(\epsilon)T_\epsilon.
	\end{align*}
	Therefore, to prove Lemma~\ref{resolventlemma2}, it is sufficient to prove that
	\begin{align}
	\label{tobeproved1}
	||\tilde{S}(\epsilon)-\tilde{S}^0(\epsilon)||_{L^2(\mathbb{R}^d)\to L^2(\mathbb{R}^d)}={O}\left(\dfrac{1}{\epsilon}\right).
	\end{align}
	
	By making $\mathcal{O}$ smaller if required~(see~\ref{assumptions2:1} and~\ref{assumptions2:2}), we may assume that
	\begin{align*}
		2(\tilde{\lambda}_{m}(\eta)-\tilde{\lambda}_0)\geq \tilde{{B}}_0(\eta-\eta_0)^2,~\eta\in\mathcal{O},\quad \mbox{ and }\\
		2(\tilde{\lambda}_{m+1}(\eta)-\tilde{\lambda}_1)\geq \tilde{{B}}_1(\eta-\eta_1)^2,~\eta\in\mathcal{O}.
	\end{align*}
	 Let $\chi$ be the characteristic function of $\mathcal{O}$, then the projections $F=\tilde{\mathcal{B}}_m^*\,\chi\,\tilde{\mathcal{B}}_m+\tilde{\mathcal{B}}_{m+1}^*\,\chi\,\tilde{\mathcal{B}}_{m+1}$ and $F^\perp=I-F$ commute with $\tilde{\mathcal{A}}$.
	
	Now, observe that
	\begin{align*}
		||\tilde{S}(\epsilon)-\tilde{S}^0(\epsilon)||_{L^2\to L^2}&=||\tilde{S}(\epsilon)F^{\perp}+\tilde{S}(\epsilon)F-\tilde{S}^0(\epsilon)F-\tilde{S}^0(\epsilon)F^\perp||_{L^2\to L^2}\notag\\
		&\leq||\tilde{S}(\epsilon)F^{\perp}||_{L^2\to L^2}+||\tilde{S}(\epsilon)F-\tilde{S}^0(\epsilon)F||_{L^2\to L^2}+||\tilde{S}^0(\epsilon)F^\perp||_{L^2\to L^2}
	\end{align*}
	
    Thus, in order to prove~\eqref{tobeproved1}, it is sufficient to prove the following:
	\begin{align}
	||\tilde{S}(\epsilon)F^{\perp}||_{L^2(\mathbb{R}^d)\to L^2(\mathbb{R}^d)}={O}(1),\label{tobeproved3}\\
	||\tilde{S}^0(\epsilon)F^\perp||_{L^2(\mathbb{R}^d)\to L^2(\mathbb{R}^d)}={O}(1),\label{tobeproved4}
	\end{align}\vspace{-1.5cm}\begin{align}
	||\tilde{S}(\epsilon)F-\tilde{S}^0(\epsilon)F||_{L^2(\mathbb{R}^d)\to L^2(\mathbb{R}^d)}={O}\left(\dfrac{1}{\epsilon}\right)	\label{tobeproved2}.
	\end{align}
	
	{\it Proof of~\eqref{tobeproved3}}: Notice that the Bloch wave decomposition of $\tilde{S}(\epsilon)$ is given by 	\begin{align*}\tilde{S}(\epsilon)=\sum_{l=1}^{\infty}\tilde{\mathcal{B}}_l^* \left(\tilde{\lambda}_l-\tilde{\lambda}_0+\epsilon^2\varkappa^2\right)^{-1}\tilde{\mathcal{B}}_l.\end{align*}
	
	We may write,
	\begin{align*}
		\tilde{S}(\epsilon)&=\tilde{S}(\epsilon)F+\tilde{S}(\epsilon)F^\perp,&
	\end{align*}
	where 
	\begin{align*}
		\tilde{S}(\epsilon)F=\tilde{\mathcal{B}}_m^* \left(\tilde{\lambda}_m-\tilde{\lambda}_0+\epsilon^2\varkappa^2\right)^{-1}\chi \tilde{\mathcal{B}}_m+\tilde{\mathcal{B}}_{m+1}^* \left(\tilde{\lambda}_{m+1}-\tilde{\lambda}_0+\epsilon^2\varkappa^2\right)^{-1}\chi \tilde{\mathcal{B}}_{m+1},
	\end{align*}
	
	and
	\begin{align}\label{awayfromedge1}
		\tilde{S}(\epsilon)F^\perp&=\sum_{l\neq m,m+1}\tilde{\mathcal{B}}_l^* \left(\tilde{\lambda}_l-\tilde{\lambda}_0+\epsilon^2\varkappa^2\right)^{-1} \tilde{\mathcal{B}}_l+\tilde{\mathcal{B}}_m^*\left(\tilde{\lambda}_m-\tilde{\lambda}_0+\epsilon^2\varkappa^2\right)^{-1}\left(1-\chi\right) \tilde{\mathcal{B}}_m\notag\\
		&\quad\quad+\tilde{\mathcal{B}}_{m+1}^*\left(\tilde{\lambda}_{m+1}-\tilde{\lambda}_0+\epsilon^2\varkappa^2\right)^{-1}\left(1-\chi\right)  \tilde{\mathcal{B}}_{m+1}.
	\end{align}
	
	To prove~\eqref{tobeproved3}, notice that in the first term of~\eqref{awayfromedge1}, the sum does not include indices $m$ and $m+1$, therefore, the Bloch eigenvalues $\tilde{\lambda}_l$ are bounded away from the spectral edge $\tilde{\lambda}_0$, uniformly in $\epsilon$ and hence, the expression $\left(\tilde{\lambda}_l-\tilde{\lambda}_0+\epsilon^2\varkappa^2\right)^{-1}$ is bounded independent of $\epsilon$, for $l\neq m, m+1$. Due to the non-degeneracy conditions assumed in~\ref{assumptions2:1} and~\ref{assumptions2:2}, the Bloch eigenvalues $\tilde{\lambda}_m$ and $\tilde{\lambda}_{m+1}$ are bounded away from $\tilde{\lambda}_0$ outside $\mathcal{O}$, independent of $\epsilon$. Hence, the last two terms in~\eqref{awayfromedge1} are bounded independent of $\epsilon$.
	
	{\it Proof of~\eqref{tobeproved4}}: Similarly, we may write
	\begin{align*}
		\tilde{S}^0(\epsilon)&=\tilde{S}^0(\epsilon)F+\tilde{S}^0(\epsilon)F^\perp,\quad\mbox{where}\\
		\tilde{S}^0(\epsilon)F&=|Y|[\tilde{\psi}_m]\left(\tilde{{B}}_0\nabla^2+\epsilon^2\varkappa^2I\right)^{-1}\chi[\overline{\tilde{\psi}_m}]+|Y|[\tilde{\psi}_{m+1}]\left(\tilde{{B}}_1\nabla^2+\epsilon^2\varkappa^2I\right)^{-1}\chi[\overline{\tilde{\psi}_{m+1}}],\quad\mbox{and},
	\end{align*}
	\begin{align}\label{awayfromedge2}
		\tilde{S}^0(\epsilon)F^\perp&=|Y|[\tilde{\psi}_m]\left(\tilde{{B}}_0\nabla^2+\epsilon^2\varkappa^2I\right)^{-1}\left(1-\chi\right)[\overline{\tilde{\psi}_m}]&\notag\\
		&\quad+|Y|[\tilde{\psi}_{m+1}]\left(\tilde{{B}}_1\nabla^2+\epsilon^2\varkappa^2I\right)^{-1}\left(1-\chi\right)[\overline{\tilde{\psi}_{m+1}}].&
	\end{align}
	$\tilde{S}^0(\epsilon)F^\perp$ may be further written as 
	\begin{align}
		\tilde{S}^0(\epsilon)F^\perp&=|Y|[\tilde{\phi}_m]\mathcal{F}^{-1}\left(\tilde{{B}}_0(\eta-\eta_0)^2+\epsilon^2\varkappa^2I\right)^{-1}\left(1-\chi\right)\mathcal{F}[\overline{\tilde{\phi}_m}]\notag\\
	&\quad+|Y|[\tilde{\phi}_{m+1}]\mathcal{F}^{-1}\left(\tilde{{B}}_1(\eta-\eta_1)^2+\epsilon^2\varkappa^2I\right)^{-1}\left(1-\chi\right)\mathcal{F}[\overline{\tilde{\phi}_{m+1}}].\notag
	\end{align}
	
	The proof of~\eqref{tobeproved4} follows from the positive-definiteness of $\tilde{B}_0$ and $\tilde{B}_1$ assumed in~\ref{assumptions2:1} and~\ref{assumptions2:2}, which makes the operator norm of the terms in~\eqref{awayfromedge2} independent of $\epsilon$. Now, it only remains to prove~\eqref{tobeproved2}.

	{\it Proof of~\eqref{tobeproved2}}: Write $\tilde{S}(\epsilon)F=S_0+S_1$, where, for $j=0,1$,
	\begin{align}\label{res1}
	S_j\coloneqq\tilde{\mathcal{B}}_{m+j}^*\left(\tilde{\lambda}_{m+j}-\tilde{\lambda}_0+\epsilon^2\varkappa^2\right)^{-1}\,\chi\, \tilde{\mathcal{B}}_{m+j}\notag\\
	=X_{m+j}^* \left(\tilde{\lambda}_{m+j}-\tilde{\lambda}_0+\epsilon^2\varkappa^2\right)^{-1} X_{m+j},
	\end{align}
	and, for $j=0,1$,
	\begin{align*}
		\left(X_{m+j} u\right)(\eta)=\int_{\mathbb{R}^d}\chi\overline{\tilde{\psi}_{m+j}}(y,\eta)u(y)~dy\quad\mbox{and}\quad\left(X_{m+j}^* u\right)(y)=\int_{\mathbb{Y}^{'}}\chi\tilde{\psi}_{m+j}(y,\eta)u(y)~d\eta.
	\end{align*}
	
	Write $\tilde{S}^0(\epsilon)F=S_0^0+S_1^0$, where, for $j=0,1$,
	\begin{align}\label{res2}
		S_j^0&=|Y|[\tilde{\phi}_{m+j}]\mathcal{F}^{-1}\left(\tilde{{B}}_j(\eta-\eta_j)^2+\epsilon^2\varkappa^2I\right)^{-1}\left(\chi\right)\mathcal{F}[\overline{\tilde{\phi}_{m+j}}]\notag\\
		&= (X^0_{m+j})^*\left(\tilde{{B}}_j(\eta-\eta_j)^2+\epsilon^2\varkappa^2I\right)^{-1}X^0_{m+j},
	\end{align}
	and, for $j=0,1$,
	\begin{align*}
		\left(X^0_{m+j} u\right)(\eta)=\int_{\mathbb{R}^d}\chi e^{-iy\cdot\eta}\overline{\tilde{\phi}_{m+j}}(y,\tilde{\eta}_j)u(y)~dy\quad\mbox{and}\quad\left(X^0_{m+j} u\right)^*(\eta)=\int_{\mathbb{R}^d}\chi e^{-iy\cdot\eta}\tilde{\phi}_{m+j}(y,\tilde{\eta}_j)u(y)~dy.
	\end{align*}
	
	Observe that, 
	\begin{align*}
		||\tilde{S}(\epsilon)F-\tilde{S}^0(\epsilon)F||_{L^2(\mathbb{R}^d)\to L^2(\mathbb{R}^d)}\leq||S_0-S_0^0||_{L^2(\mathbb{R}^d)\to L^2(\mathbb{R}^d)}+||S_1-S_1^0||_{L^2(\mathbb{R}^d)\to L^2(\mathbb{R}^d)}.
	\end{align*}
	
	Therefore, to prove~\eqref{tobeproved2}, it remains to prove that for $j=0,1$,
	\begin{align*}
	||{S}_j-{S}_j^0||_{L^2(\mathbb{R}^d)\to L^2(\mathbb{R}^d)}={O}\left(\dfrac{1}{\epsilon}\right).
	\end{align*}
	
	where $S_j, S_j^0$ are defined in~\eqref{res1},~\eqref{res2}.
	
	Consider,
	\begin{align*}
		\epsilon||{S}_0-{S}_0^0||&=\epsilon||X_{m}^* [\left(\tilde{\lambda}_{m}-\tilde{\lambda}_0+\epsilon^2\varkappa^2\right)^{-1}] X_{m}-(X^0_{m})^*\left(\tilde{{B}}_0(\eta-\eta_0)^2+\epsilon^2\varkappa^2I\right)^{-1}X^0_{m}||.
	\end{align*}
	Therefore, \begin{align}\label{final1}
		\epsilon||{S}_0-{S}_0^0||& \leq \epsilon||X_{m}^* [\left(\tilde{\lambda}_{m}-\tilde{\lambda}_0+\epsilon^2\varkappa^2\right)^{-1}] X_{m}-X_{m}^*\left(\tilde{{B}}_0(\eta-\eta_0)^2+\epsilon^2\varkappa^2I\right)^{-1}X_{m}||&\notag\\
		&\qquad+\epsilon||X_{m}^*\left(\tilde{{B}}_0(\eta-\eta_0)^2+\epsilon^2\varkappa^2I\right)^{-1}X_{m}-(X^0_{m})^*\left(\tilde{{B}}_0(\eta-\eta_0)^2+\epsilon^2\varkappa^2I\right)^{-1}X^0_{m}||.&
	\end{align}
	
	The first of the two terms on the right hand side (RHS) in the inequality~\eqref{final1} is estimated by using the following chain of inequalities. 
	\begin{align*}
		&\epsilon|\left(\tilde{\lambda}_{m}-\tilde{\lambda}_0+\epsilon^2\varkappa^2\right)^{-1} -\left(\tilde{{B}}_0(\eta-\eta_0)^2+\epsilon^2\varkappa^2I\right)^{-1}|&\notag\\
		&\qquad\leq c\epsilon|\eta-\eta_0|^3\left(\tilde{\lambda}_{m}-\tilde{\lambda}_0+\epsilon^2\varkappa^2\right)^{-1}\left(\tilde{{B}}_0(\eta-\eta_0)^2+\epsilon^2\varkappa^2I\right)^{-1}&\notag\\
		&\qquad\leq\left(c|\eta-\eta_0|^2\left(\tilde{{B}}_0(\eta-\eta_0)^2\right)^{-1}\right)\left(2\epsilon|\eta-\eta_0|\left(\tilde{{B}}_0(\eta-\eta_0)^2+\epsilon^2\varkappa^2I\right)^{-1}\right)\leq C_1.&
	\end{align*}
	
	The proof of the boundedness of the second term on the RHS in inequality~\eqref{final1} hinges on the analyticity of the Bloch eigenfunctions, and may be found in~\cite{Birman2006}. Finally, consider
	\begin{align*}
		\epsilon||{S}_1-{S}_1^0||&=\epsilon||X_{m+1}^* [\left(\tilde{\lambda}_{m+1}-\tilde{\lambda}_0+\epsilon^2\varkappa^2\right)^{-1}] X_{m+1}-(X^0_{m+1})^*\left(\tilde{{B}}_1(\eta-\eta_0)^2+\epsilon^2\varkappa^2I\right)^{-1}X^0_{m+1}||.
	\end{align*}
	Therefore,
	\begin{align}\label{final2}
		\epsilon||{S}_1-{S}_1^0||&\leq \epsilon||X_{m+1}^* [\left(\tilde{\lambda}_{m+1}-\tilde{\lambda}_0+\epsilon^2\varkappa^2\right)^{-1}] X_{m+1}-X_{m+1}^*\left(\tilde{{B}}_1(\eta-\eta_0)^2+\epsilon^2\varkappa^2I\right)^{-1}X_{m+1}||&\notag\\
		&~+\epsilon||X_{m+1}^*\left(\tilde{{B}}_1(\eta-\eta_0)^2+\epsilon^2\varkappa^2I\right)^{-1}X_{m+1}-(X^0_{m+1})^*\left(\tilde{{B}}_1(\eta-\eta_0)^2+\epsilon^2\varkappa^2I\right)^{-1}X^0_{m+1}||.&
	\end{align}
	
	The first of the two terms on RHS in inequality~\eqref{final2} is estimated by using the following chain of inequalities. 
	\begin{align*}
		&\epsilon|\left(\tilde{\lambda}_{m+1}-\tilde{\lambda}_0+\epsilon^2\varkappa^2\right)^{-1} -\left(\tilde{{B}}_1(\eta-\eta_1)^2+\epsilon^2\varkappa^2I\right)^{-1}|&\notag\\
		&\qquad\leq\epsilon|\left(\tilde{\lambda}_{m+1}-\tilde{\lambda}_1+\epsilon^2\varkappa^2\right)^{-1} -\left(\tilde{{B}}_1(\eta-\eta_1)^2+\epsilon^2\varkappa^2I\right)^{-1}|&\notag\\
		&\qquad\leq c\epsilon|\eta-\eta_1|^3\left(\tilde{\lambda}_{m+1}-\tilde{\lambda}_1+\epsilon^2\varkappa^2\right)^{-1}\left(\tilde{{B}}_0(\eta-\eta_1)^2+\epsilon^2\varkappa^2I\right)^{-1}&\notag\\
		&\qquad\leq\left(c|\eta-\eta_1|^2\left(\tilde{{B}}_1(\eta-\eta_0)^2\right)^{-1}\right)\left(2\epsilon|\eta-\eta_1|\left(\tilde{{B}}_1(\eta-\eta_0)^2+\epsilon^2\varkappa^2I\right)^{-1}\right)\leq C_2.&
	\end{align*}
	
	As before, the proof of the boundedness of the second term on RHS in inequality~\eqref{final2} may be found in~\cite{Birman2006}.\end{proof}

\begin{appendix}
\renewcommand{\thm}{{\bf Theorem A.\arabic{thm}}}
\section{Perturbation Theory of holomorphic family of type $(B)$}
\label{PerturbationTheory}
In this section, we show that a perturbation in the coefficients of the operator $\mathcal{A}$ gives rise to a corresponding holomorphic family of sectorial forms of type $(a)$. Further, the selfadjointness of the forms coupled with the compactness of the resolvent for the operator family ensures that it is a selfadjoint holomorphic family of type $(B)$. For definition of these notions, see Kato~\cite{Kato1995}. 

Let $A\in M_B^{>}$ and $B=(b_{kl})$ be a symmetric matrix with $L^\infty_{\sharp}(Y,\mathbb{R})$ entries. Then, for $\sigma<\frac{\alpha}{d||B||_{L^\infty}}$, $A+\sigma B$ belongs to $M_B^{>}$, where $\alpha$ is a coercivity constant for $A$, as in~\eqref{coercivity}. For a fixed $\eta_0 \in Y^{'}$ and for $\sigma_0\coloneqq \frac{\alpha}{2d||B||_{L^\infty}}$, let us define the operator family \begin{align*}\mathcal{A}(\eta_0)(\tau)=-(\nabla+i\eta_0)\cdot(A+\tau B)(\nabla+i\eta_0),\quad\tau\in R,\end{align*} where $R=\{z\in\mathbb{C}:|\operatorname{Re}(z)|<\sigma_0,|\operatorname{Im}(z)|<\sigma_0\}.$ For real $\tau$, $-\sigma_0<\tau<\sigma_0$, $A+\tau B$ is coercive with a coercivity constant $\alpha/2$. The holomorphic family of sesquilinear forms $\mathfrak{t}(\tau)$ associated to operator $\mathcal{A}+\tau\mathcal{B}$, with the $\tau$-independent domain $\mathfrak{D}(\mathfrak{t}(\tau))=H^1_\sharp(Y)$, is defined as \begin{align*}
\mathfrak{t}(\tau)[u,v] &\coloneqq \int_Y \left(a_{kl}(y)+\tau b_{kl}(y)\right)\frac{\partial u}{\partial y_l}\frac{\partial \overline{v}}{\partial y_k}~dy+i\eta_{0,l}\int_Y \left(a_{kl}(y)+\tau b_{kl}(y)\right) u\frac{\partial \overline{v}}{\partial y_k}~dy\\
&\quad-i\eta_{0,k}\int_Y \left(a_{kl}(y)+\tau b_{kl}(y)\right) \overline{v} \frac{\partial u}{\partial y_l}~dy+\eta_{0,l}\eta_{0,k}\int_Y \left(a_{kl}(y)+\tau b_{kl}(y)\right) u\overline{v}~dy,
\end{align*}

where $\eta_0\coloneqq(\eta_{0,1},\eta_{0,2},\ldots,\eta_{0,d})$ and summation over repeated indices is assumed.

\begin{thm}
	$\mathfrak{t}(\tau)$ is a holomorphic family of type $(a)$.
\end{thm}

\begin{proof}
	The quadratic form associated with $\mathfrak{t}(\tau)$ is as follows:
	\begin{align*}
	\mathfrak{t}(\tau)[u] & \coloneqq \int_Y \left(a_{kl}(y)+\tau b_{kl}(y)\right)\frac{\partial u}{\partial y_l}\frac{\partial \overline{u}}{\partial y_k}~dy+i\eta_{0,l}\int_Y \left(a_{kl}(y)+\tau b_{kl}(y)\right) u\frac{\partial \overline{u}}{\partial y_k}~dy
	\\
	& \quad-i\eta_{0,k}\int_Y \left(a_{kl}(y)+\tau b_{kl}(y)\right) \overline{u} \frac{\partial u}{\partial y_l}~dy+\eta_{0,k}\eta_{0,l}\int_Y \left(a_{kl}(y)+\tau b_{kl}(y)\right) u\overline{u}~dy.
	\end{align*}
	
	{\bf $(i)$} $\mathfrak{t}(\tau)$ is sectorial.
	
	Let us write $\tau=\rho+i\gamma$, then the quadratic form $\mathfrak{t}(\tau)$ can be written as the sum of its real and imaginary parts:
	\begin{align*}
	\mathfrak{t}(\tau)=	\Re{\mathfrak{t}(\tau)[u]}+	i\Im{\mathfrak{t}(\tau)[u]}
	\end{align*}
	
	where the real part is
	\begin{align}\label{realpart}
	\Re{\mathfrak{t}(\tau)[u]} & \coloneqq \int_Y \left(a_{kl}(y)+\rho b_{kl}(y)\right)\frac{\partial u}{\partial y_l}\frac{\partial \overline{u}}{\partial y_k}~dy+i\eta_{0,l}\int_Y \left(a_{kl}(y)+\rho b_{kl}(y)\right) u\frac{\partial \overline{u}}{\partial y_k}~dy\notag
	\\
	&\quad -i\eta_{0,k}\int_Y \left(a_{kl}(y)+\rho b_{kl}(y)\right) \overline{u} \frac{\partial u}{\partial y_l}~dy+\eta_{0,k}\eta_{0,l}\int_Y \left(a_{kl}(y)+\rho b_{kl}(y)\right) u\overline{u}~dy,
	\end{align}
	
	and the imaginary part is
	\begin{align}\label{imaginarypart}
	\Im{\mathfrak{t}(\tau)[u]} & \coloneqq \int_Y \left(\gamma b_{kl}(y)\right)\frac{\partial u}{\partial y_l}\frac{\partial \overline{u}}{\partial y_k}~dy+2\Im{\left(\eta_{0,l}\int_Y \left(\gamma b_{kl}(y)\right) u\frac{\partial \overline{u}}{\partial y_k}~dy\right)}\notag
	\\
	&\quad +\eta_{0,k}\eta_{0,l}\int_Y \left(\gamma b_{kl}(y)\right) u\overline{u}~dy.
	\end{align}
	
	The real part~\eqref{realpart} of ${\mathfrak{t}(\tau)[u]}$ may also be written as \begin{align}\label{realpart2}
	\Re{\mathfrak{t}(\tau)[u]} & \coloneqq \int_Y \left(a_{kl}(y)+\rho b_{kl}(y)\right)\frac{\partial u}{\partial y_l}\frac{\partial \overline{u}}{\partial y_k}~dy+2\Re{\left(i\eta_{0,l}\int_Y \left(a_{kl}(y)+\rho b_{kl}(y)\right) u\frac{\partial \overline{u}}{\partial y_k}~dy\right)}\notag
	\\
	&\quad +\eta_{0,k}\eta_{0,l}\int_Y \left(a_{kl}(y)+\rho b_{kl}(y)\right) u\overline{u}~dy.
	\end{align}
	
	The first term in~\eqref{realpart2} is estimated from below as follows:
	\begin{equation}
	\int_Y \left(a_{kl}(y)+\rho b_{kl}(y)\right)\frac{\partial u}{\partial y_l}\frac{\partial \overline{u}}{\partial y_k}~dy \geq \frac{\alpha}{2}\int_{Y}|\nabla u|^2~dy.\label{eq:111}
	\end{equation}
	
	The second term in~\eqref{realpart2} may be bounded from above as follows:
	\begin{align}
	\Re{\left(i\eta_{0,l}\int_Y \left(a_{kl}(y)+\rho b_{kl}(y)\right) u\frac{\partial \overline{u}}{\partial y_k}~dy\right)} & \leq  C_1||u||_{L^2_{\sharp}(Y)}||\nabla u||_{L^2_{\sharp}(Y)}\notag
	\\
	& \leq  C_1C_*||u||^2_{L^2_{\sharp}(Y)}+\frac{C_1}{C_*}||\nabla u||^2_{L^2_{\sharp}(Y)}\notag
	\\
	& =  C_2||u||^2_{L^2_{\sharp}(Y)}+\frac{\alpha}{4}||\nabla u||^2_{L^2_{\sharp}(Y)},\label{eq:222}
	\end{align}
	
	where $C_*=\frac{4C_1}{\alpha}$ and $C_1, C_2$ are some constants independent of $u$ and $\rho$.
	
	The last term in~\eqref{realpart2} is estimated as
	\begin{align}
	\eta_{0,k}\eta_{0,l}\int_Y \left(a_{kl}(y)+\rho b_{kl}(y)\right) u\overline{u}~dy \leq C_3 ||u||^2_{L^2_{\sharp}(Y)},\label{eq:333}
	\end{align} for some $C_3>0$.
	
	Finally, combining \eqref{eq:111}, \eqref{eq:222} and \eqref{eq:333}, we obtain
	\begin{align}
	\Re{\mathfrak{t}(\tau)[u]} \geq \frac{\alpha}{4}||u||^2_{H^1_{\sharp}(Y)}-C_4||u||^2_{L^2_{\sharp}(Y)},\label{eq:444}
	\end{align} for some $C_4>0$.

	Estimating the imaginary part~\eqref{imaginarypart} from above, we obtain
	\begin{align}
	|\Im{\mathfrak{t}(\tau)[u]}| \leq C_5||\nabla u||^2_{L^2_\sharp(Y)}+C_6||u||^2_{L^2_\sharp(Y)},\label{eq:999}
	\end{align} for some positive $C_5, C_6$.
	
	Now, choose a scalar $C_7$ so that $C_7=\frac{4C_5}{\alpha}$.
	
	The inequality \eqref{eq:444} may be written as
	\begin{align}
	\Re{\mathfrak{t}(\tau)[u]} + C_4||u||^2_{L^2_{\sharp}(Y)} + \frac{C_6}{C_7}||u||^2_{L^2_{\sharp}(Y)} \geq \frac{\alpha}{4}||u||^2_{H^1_{\sharp}(Y)}+\frac{C_6}{C_7}||u||^2_{L^2_{\sharp}(Y)}.\label{eq:555}
	\end{align}
	
	Now, we define a new quadratic form $\tilde{\mathfrak{t}}[u]\coloneqq \mathfrak{t}[u]+(C_4+\frac{C_6}{C_7})||u||^2_{L^2_\sharp{Y}},$ then inequality \eqref{eq:555} becomes
	\begin{align}
	\Re{\tilde{\mathfrak{t}}(\tau)[u]} \geq \frac{\alpha}{4}||u||^2_{H^1_{\sharp}(Y)}+\frac{C_6}{C_7}||u||^2_{L^2_{\sharp}(Y)}.\label{eq:666}
	\end{align}
	
	This may be further written as
	\begin{align}
	\Re{\tilde{\mathfrak{t}}(\tau)[u]} - \frac{\alpha}{4}||u||^2_{L^2_\sharp(Y)}  \geq  \frac{\alpha}{4}||\nabla u||^2_{L^2_{\sharp}(Y)}+\frac{C_6}{C_7}||u||^2_{L^2_{\sharp}(Y)}.\label{eq:777}
	\end{align}
	
	On multiplying throughout by $C_7$, the inequality \eqref{eq:777} becomes
	\begin{align}
	C_7\left\{\Re{\tilde{\mathfrak{t}}(\tau)[u]} - \frac{\alpha}{4}||u||^2_{L^2_\sharp(Y)}\right\}  \geq  C_5||\nabla u||^2_{L^2_{\sharp}(Y)}+{C_6}||u||^2_{L^2_{\sharp}(Y)}.\label{eq:888}
	\end{align}
	
	Since $\Im{\tilde{\mathfrak{t}}[u]}=\Im{\mathfrak{t}[u]}$, combining the inequalities \eqref{eq:999} and \eqref{eq:888}, we obtain
	\begin{align*}
	|\Im{\tilde{\mathfrak{t}}(\tau)[u]}| \leq  C_7\left\{\Re{\tilde{\mathfrak{t}}(\tau)[u]} - \frac{\alpha}{4}||u||^2_{L^2_\sharp(Y)}\right\}.
	\end{align*}
	
	This proves that the form $\tilde{\mathfrak{t}}$ is sectorial. However, the property of sectoriality is invariant under a shift. Therefore, $\mathfrak{t}$ is sectorial, as well.
	
	{\bf $(ii)$} $\mathfrak{t}(\tau)$ is closed.
	
	This follows from the inequality \eqref{eq:666}. If $u_n\xrightarrow{{\mathfrak{t}-convergence}}u$ then $\Re{\mathfrak{t}[u_n-u_m]}\to 0$ as $n,m\to\infty$. By \eqref{eq:666}, $(u_n)$ is a Cauchy sequence in $H^1_\sharp(Y)$. By completeness, there is $v\in H^1_\sharp(Y)$ to which the sequence converges. However, $\mathfrak{t}$-convergence implies $L^2$ convergence, and therefore, $u=v$. Clearly, $\mathfrak{t}[u_n-u]\to 0$.
	
	{\bf $(iii)$} $\mathfrak{t}(\tau)$ is a holomorphic family of type $ (a) $.
	
	We have proved that $\mathfrak{t}(\tau)[u]$ is sectorial and closed. It remains to prove that the form is holomorphic. This is easily done since $\mathfrak{t}(\tau)[u]$ is linear in $\tau$ for each fixed $u\in H^1_\sharp(Y)$.
\end{proof}

The first representation theorem of Kato ensures that there exists a unique m-sectorial operator with domain contained in $H^1_\sharp(Y)$ associated with each $\mathfrak{t}(\tau)$. A proof may be found in~\cite[p.322]{Kato1995}.  The family of such operators associated with a holomorphic family of sesquilinear forms of type $(a)$ is called a holomorphic family of type $( B )$. The aforementioned m-sectorial operator is given by
\begin{align*}\mathcal{A}(\eta_0)(\tau)=-(\nabla+i\eta_0)(A+\tau B)(\nabla+i\eta_0).\end{align*} It follows from the symmetry of the matrix $A+\tau B$ that the family $\mathcal{A}(\eta_0)(\tau)$ is a selfadjoint holomorphic family of type $(B)$. Moreover, by the compact embedding of $H^1_\sharp(Y)$ in $L^2_\sharp(Y)$, the operator $\mathcal{A}(\eta_0)(\tau)+C^*I$ has compact resolvent for each $\tau\in R$ for some appropriate constant $C^*$, independent of $\tau\in R$.

Hence, by Kato-Rellich Theorem, there exists a complete orthonormal set of eigenvectors associated with the operator family $\mathbb{A}(\eta_0)(\tau)$ which are analytic for the whole interval $-\sigma_0<\tau<\sigma_0$.
\end{appendix}
\bibliographystyle{plain}
\bibliography{mylit.bib}
\end{document}